\documentclass{article}

\usepackage{microtype}
\usepackage{graphicx}
\usepackage{subfigure}
\usepackage{booktabs} 

\usepackage{hyperref}

\usepackage[accepted]{icml2023}

\usepackage{amsmath}
\usepackage{amssymb}
\usepackage{mathtools}
\usepackage{amsthm}

\theoremstyle{plain}
\newtheorem{theorem}{Theorem}[section]
\newtheorem{proposition}[theorem]{Proposition}
\newtheorem{lemma}[theorem]{Lemma}

\theoremstyle{definition}

\theoremstyle{remark}
\newtheorem{remark}[theorem]{Remark}

\usepackage[textsize=tiny]{todonotes}

\usepackage{nicefrac}
\usepackage{pythonhighlight}
\usepackage{listings}
\usepackage{adjustbox}
\usepackage{enumerate}
\usepackage{mathtools}
\mathtoolsset{showonlyrefs}
\usepackage{xcolor}

\newcommand{\R}{\mathbb{R}}

\newcommand{\cG}{\mathcal{G}}
\newcommand{\cE}{\mathcal{E}}

\newcommand{\norm}[1]{\left\lVert #1 \right\rVert}
\newcommand{\normsq}[1]{\norm{#1}^{2}}
\newcommand{\sprod}[2]{\left\langle #1, #2 \right\rangle}

\newcommand{\gra}[1][]{\operatorname{gra}}
\newcommand{\nx}{\nabla_{x}}
\newcommand{\ny}{\nabla_{y}}
\newcommand{\ph}{\, \cdot \,}
\newcommand{\minus}{\scalebox{0.75}[1.0]{$-$}}

\newcommand{\Gap}{\mathtt{Gap}}
\newcommand{\Res}{\mathtt{Res}}

\DeclareMathOperator*{\argmin}{arg\,min}

\begin{document}

\twocolumn[
\icmltitle{A Fast Optimistic Method for Monotone Variational Inequalities}



\icmlsetsymbol{equal}{*}

\begin{icmlauthorlist}
\icmlauthor{Michael Sedlmayer}{DS}
\icmlauthor{Dang-Khoa Nguyen}{Math}
\icmlauthor{Radu Ioan Bo\c{t}}{DS,Math}
\end{icmlauthorlist}

\icmlaffiliation{DS}{Research Network Data Science, University of Vienna, Vienna, Austria}
\icmlaffiliation{Math}{Faculty of Mathematics, University of Vienna, Vienna, Austria}

\icmlcorrespondingauthor{Michael Sedlmayer}{michael.sedlmayer@univie.ac.at}

\icmlkeywords{Variational Inequalities, Acceleration, Convergence Rates, ICML}

\vskip 0.3in
]



\printAffiliationsAndNotice{}  

\begin{abstract}
	We study monotone variational inequalities that can arise as optimality conditions for constrained convex optimisation or convex-concave minimax problems and propose a novel algorithm that uses only one gradient/operator evaluation and one projection onto the constraint set per iteration. The algorithm, which we call \emph{fOGDA-VI}, achieves a $o(\nicefrac{1}{k})$ rate of convergence in terms of the restricted gap function as well as the natural residual for the \emph{last iterate}. Moreover, we provide a convergence guarantee for the sequence of iterates to a solution of the variational inequality. These are the best theoretical convergence results for numerical methods for (only) monotone variational inequalities reported in the literature. To empirically validate our algorithm we investigate a two-player matrix game with mixed strategies of the two players. Concluding, we show promising results regarding the application of fOGDA-VI to the training of generative adversarial nets.
\end{abstract}

\section{Introduction}
Variational inequalities are fundamental models in various fields such as optimisation, e.g., when determining primal-dual pairs of optimal solutions of constrained convex optimisation problems~\cite{BauschkeCombettes}, economics, game theory~\cite{GameTheory}, or partial differential equations. Recently, they have attracted particularly significant attention in the area of machine learning due to the fundamental role they play, for instance, in multi agent reinforcement learning~\cite{omidshafiei2017deep}, robust adversarial learning~\cite{madry2018towards} and the training of generative adversarial networks (GANs)~\cite{GAN, GAN2}.

\subsection{Problem Setting}
In the following we consider $\R^{d}$ with its standard inner product denoted by $ \sprod{\ph}{\ph} $ and induced norm $ \norm{\ph} $. Let $F: \R^{d} \to \R^{d}$ be a monotone operator, i.e.,
\begin{equation}
	\sprod{F(w) - F(z)}{w - z} \geq 0
	\quad \forall w,z \in \R^{d},
\end{equation}
which is also $L$-Lipschitz continuous, i.e.,
\begin{equation}
	\norm{F(w) - F(z)} \leq L \norm{w - z}
	\quad \forall w,z \in \R^{d}.
\end{equation}
Furthermore, let $C$ be a nonempty closed convex subset of $ \R^{d} $. Then the (strong) classical \emph{variational inequality} problem consists of finding $ z^{\ast} \in C $ such that
\begin{equation}\label{intro:pb:vi}
	\sprod{F(z^{\ast})}{z - z^{\ast}} \geq 0
	\quad \forall z \in C.
\end{equation}
For the following considerations we assume that the solution set of \eqref{intro:pb:vi} is nonempty, i.e., $ \Omega := \{ z^{\ast} \in C \mid \sprod{F(z^{\ast})}{z - z^{\ast}} \geq 0 \quad \forall z \in C \} \neq \emptyset $.

Note that in the case of $ F $ being monotone and continuous, the above \emph{strong} formulation is equivalent to the following problem,
\begin{equation}\label{eq:VI-weak}
	\sprod{F(z)}{z - z^{\ast}} \geq 0
	\quad \forall z \in C,
\end{equation}
which is known as the \emph{weak} version of the variational inequality.
Writing $ N_{C}(z) := \{ w \in \R^{d} \mid \sprod{v - z}{w} \leq 0 \quad  \forall v \in C \} $, for $ z \in C $, and $ N_{C}(z) := \emptyset$, for $z \notin C$, to denote the normal cone of $ C $, condition \eqref{intro:pb:vi} is equivalent to the following \emph{monotone inclusion}, where want to find $ z^{\ast} \in \R^{d} $ such that
\begin{equation}\label{intro:pb:mi}
	0 \in F(z^{\ast}) + N_{C}(z^{\ast}).
\end{equation}

\subsection{Contribution}
We introduce an accelerated first order method for solving the constrained variational inequality problem~\eqref{intro:pb:vi} that uses a single operator evaluation and a single projection in each iteration.
Our proposed algorithm, called fOGDA-VI, exhibits a $ o(\nicefrac{1}{k}) $ rate of convergence for the last iterate which is better than the $ \mathcal{O}(\nicefrac{1}{k}) $ results for other accelerated algorithms.
Moreover, fOGDA-VI exhibits convergence of the generated sequence to a solution of the variational inequality under investigation, which is not necessarily the case for other accelerated methods~\cite{cai2022acceleratedalgorithms,cai2022acceleratedsingle} that have been proposed for~\eqref{intro:pb:vi}.

\subsection{Overview}
This paper is structured as follows. In Section~\ref{sec:VI} we discuss suitable convergence measures and (accelerated) solution methods for monotone variational inequalities governed by a monotone and Lipschitz operator. The algorithm fOGDA-VI and the accompanying convergence results are presented in Section~\ref{sec:main}, which is followed by illustrations of the empirical performance of the proposed method when solving two-player matrix games and in the training of GANs in Section~\ref{sec:num}.

\section{Solving Variational Inequalities}\label{sec:VI}
In this section we recall appropriate measures of convergence for solution methods for monotone variational inequalities and provide an overview on the most important solution methods from the literature, both nonaccelerated and accelerated ones, for solving~\eqref{intro:pb:vi}.

\subsection{Convergence Measures}
We start with presenting three suitable measures that are commonly used to judge the quality of prospective solutions.

\paragraph{Restricted gap function}
For $z^{\ast} \in \Omega, z_0 \in \R^{d}$ and $\delta(z_{0}) := \norm{z^{\ast} - z_{0}} $, the restricted \textit{gap} function associated with the variational inequality \eqref{intro:pb:vi} is defined as
\begin{equation}
	\Gap(z) := \sup_{w \in C \cap \mathbb{B}(z^{\ast}; \delta(z_{0}))} \; \sprod{F(w)}{z - w} \geq 0.
\end{equation}
It is also known as \emph{merit} function~\cite{nesterov2007dual} and it measures how much the statement of~\eqref{eq:VI-weak} is violated.

In the above definition, $ \mathbb{B}(z; \delta) := \left\{ w \in \R^{d} \mid \norm{w - z} \leq \delta \right\}$ denotes the closed ball centred at $ z \in \R^{d} $ with radius $ \delta > 0 $. The restriction of the supremum to a bounded set, particularly the ball $\mathbb{B} (z^{\ast} ; \delta(z_{0}))$ in our case, is essential to avoid an infinitely large gap when $C$ is unbounded.

\paragraph{Tangent residual}
Another quantity that can be used to measure the quality of a solution  candidate with respect to the variational inequality \eqref{intro:pb:vi} is based on the observation that the latter it is equivalent to the monotone inclusion \eqref{intro:pb:mi}. The so-called \emph{tangent residual} is given by
\begin{equation*}
	r(z) := \inf_{\zeta \in N_{C} (z)} \; \left\lVert F(z) + \zeta \right\rVert .
\end{equation*}
In a straightforward way this quantity extends the usual measure $ \norm{F(z)} $ in the unconstrained setting of monotone equations, where the goal is to find $ z^{\ast} \in \R^{d} $ such that
\begin{equation}\label{intro:pb:eq}
	F(z^{\ast}) = 0,
\end{equation}
to variational inequalities by measuring the distance from $0$ to $F(z) + N_C(z)$. Note, if $z \notin C$ we have $N_{C} (z) = \emptyset$ and thus $ r(z) = +\infty $.

\paragraph{Natural residual}
Another useful convergence measure is the \emph{natural residual}, which in fact is upper bounded by the tangent residual, see Section \ref{app:prelim}. For this we write $P_{C}(z)$ to denote the projection of $z \in \R^{d}$ onto the closed convex set $C$, which is uniquely defined and given by $ P_{C}(z) = \argmin_{w \in C} \norm{w - z} $. Using the characterisation of the projection via the normal cone (see Proposition~6.46 in~\cite{BauschkeCombettes}) we observe that
\begin{equation}
	0 \in F(z^{\ast}) + N_{C}(z^{\ast})
	\hspace{0.5em} \Leftrightarrow \hspace{0.5em}
	z^{\ast} = P_{C} \left[ z^{\ast} - F(z^{\ast}) \right].
\end{equation}
This motivates to look at
\begin{equation}
	\Res(z) := \norm{z - P_{C} \left[ z - F(z) \right]},
\end{equation}
which is also known as \emph{fixed point residual}. Note, in the unconstrained case~\eqref{intro:pb:eq} the two residuals coincide
\begin{equation}
	r(z) = \Res(z) = \norm{F(z)} \quad \forall z \in \R^{d}.
\end{equation}

\subsection{Solution Methods}
In this work we are interested exclusively in first order methods that are fully splitting, i.e., algorithms that only use direct evaluations of the operator $F$ and projections onto $C$ as main building blocks. As a general Lipschitz continuous operator is not necessarily cocoercive, the simplest first order method that is splitting -- the Forward-Backward (FB) algorithm -- can not be used to solve~\eqref{intro:pb:vi}.

\subsubsection{Nonaccelerated Solution Methods}\label{sub:nonaccmeth}
\paragraph{Extragradient (EG) method}
Korpelevich~\yrcite{extragradient} and Antipin~\yrcite{antipin1976method} proposed to take a second forward evaluation of $F$ in each iteration in order to solve \eqref{intro:pb:vi}. This results in the following scheme for $ k \geq 0 $
\begin{equation}\label{algo:EG}
	\text{EG:}
	\left\lfloor
		\begin{array}{l}
			w_{k} = P_{C} \left[ z_{k} - \gamma F(z_{k}) \right] \\
			z_{k+1} = P_{C} \left[ z_{k} - \gamma F(w_{k}) \right]
		\end{array}
	\right.
\end{equation}
which converges to a solution of \eqref{intro:pb:vi} for $0 < \gamma < \nicefrac{1}{L}$.

It is known that EG converges with a rate of $\mathcal{O} (\nicefrac{1}{K})$ in terms of the restricted gap function for the \emph{averaged}, or \emph{ergodic}, iterates
\begin{equation}
	\bar{w}_{K} := \frac{1}{K} \sum_{k = 1}^{K} w_{k}
\end{equation}
in both the unconstrained~\cite{mirror-prox,nesterov2007dual,ogda-rate-1/k} and the constrained case~\cite{hsieh2019convergence}, which further seems to be optimal~\cite{ouyang2021lower}.
The \emph{best iterate} convergence in terms of the tangent residual, however, is known to yield a rate of $\mathcal{O} (\nicefrac{1}{\sqrt{K}})$~\cite{extragradient,facchinei2003finite}, i.e.,
\begin{equation}
	\min_{1 \leq k \leq K} r (z_{k}) = \mathcal{O} \left( \frac{1}{\sqrt{K}} \right)
	\quad \text{as } K \to +\infty .
\end{equation}
The more desirable \emph{last iterate} convergence rate for EG was derived only recently in the unconstrained case~\cite{gorbunov2022extragradient} which was then extended to the constrained case as well~\cite{cai2022tight}.
In fact,
\begin{equation}
	\Gap(z_{k}) = \mathcal{O} \left( \frac{1}{\sqrt{k}} \right)
	\hspace{0.5em} \text{and} \hspace{0.5em}
	r(z_{k}) = \mathcal{O} \left( \frac{1}{\sqrt{k}} \right),
\end{equation}
as $k \to + \infty$. The result for the restricted gap function~\cite{last-iterate-rate-slow} as well as for the residuals is actually tight, meaning that the convergence rate for the averaged iterates is better than for the last iterate.
Nevertheless, we emphasise that the latter one is more appealing and that the averaged iterates might still show acceptable behaviour while the actual trajectory of iterates cycles around the set of solutions~\cite{mertikopoulos2018cycles}.

\paragraph{Popov's method}
In the saddle point setting Popov~\yrcite{Popov} introduced the following algorithm which, when applied to~\eqref{intro:pb:mi}, reads for $ k \geq 1$
\begin{equation}\label{algo:OGDA}
	\text{Popov:}
	\left\lfloor
		\begin{array}{l}
			w_{k} = P_{C} \left[ z_{k} - \gamma F(w_{k-1}) \right] \\
			z_{k+1} = P_{C} \left[ z_{k} - \gamma F(w_{k}) \right]
		\end{array}
	\right.
\end{equation}
which converges to a solution of \eqref{intro:pb:vi} for $ 0 < \gamma < \nicefrac{1}{2L} $.
The update rule of \eqref{algo:OGDA} is very similar to \eqref{algo:EG} but requires $F$ to be evaluated only once per iteration. Actually, Popov differs from EG only in the first block, where $F(z_{k})$ is replaced by $F(w_{k-1})$. In the unconstrained case,~\eqref{algo:OGDA} can be written in one line, yielding a method usually known as Optimistic Gradient Descent Ascent (OGDA), a name that was coined by works on GAN training~\cite{OGDA,OGDA2}.

Given the close connection between EG and OGDA, it is not surprising that many convergence rate results hold in a similar way. In terms of the restricted gap OGDA converges like $\mathcal{O}(\nicefrac{1}{k})$ and $\mathcal{O}(\nicefrac{1}{\sqrt{k}})$ for averaged iterates~\cite{ogda-rate-1/k} and last iterates~\cite{last-iterate-rate-slow}, respectively, where the latter one is optimal. The convergence rate in terms of the residuals is $\mathcal{O}(\nicefrac{1}{\sqrt{k}})$ and it can not be improved in general~\cite{golowich2020tight,chavdarova2021last,cai2022tight}, as seen for EG.

We have seen that both EG and Popov's method require two projections in each iteration. One might think that this is necessary to obtain convergent algorithms for the variational inequality~\eqref{intro:pb:mi} when the operator $ F $ is merely Lipschitz continuous and not cocoercive. This is not the case, however, and in the following we will look at algorithms that need only one evaluation of the projection operator per iteration.

\paragraph{Forward-Backward-Forward (FBF) method}
One of these single-call projection methods is the FBF method. It was proposed by Tseng~\yrcite{tseng-fbf} and applied to~\eqref{intro:pb:mi} it iterates for $ k \geq 0 $
\begin{equation}\label{algo:Tseng}
	\text{FBF:}
	\left\lfloor
		\begin{array}{l}
			w_{k} = P_{C} \left[ z_{k} - \gamma F(z_{k}) \right] \\
			z_{k+1} = w_{k} - \gamma F(w_{k}) + \gamma F(z_{k})
		\end{array}
	\right.
\end{equation}
which converges to a solution of \eqref{intro:pb:vi} for $ 0 < \gamma < \nicefrac{1}{L} $.
Notice that in each iteration this iterative scheme performs two evaluations of $F$ along the sequences $(z_{k})_{k \geq 0}$ and $(w_{k})_{k \geq 0}$, similar to EG -- the first line of FBF and EG is even identical. In the second line, however, instead of performing another projection the forward step regarding the intermediate iterate $ w_{k} $ is corrected by the previous update $ F(z_{k}) $.
Moreover, for the unconstrained problem~\eqref{intro:pb:eq}, i.e., in the absence of projections, FBF and EG are equivalent.

\paragraph{Forward-Reflected-Backward (FRB) method}
Another single-call projection method that even requires only one evaluation of $ F $ like the basic FB algorithm was proposed by Malitsky and Tam~\yrcite{yura-mat-reflection}. The FRB algorithm is given for $ k \geq 1$ by
\begin{equation}\label{algo:Malitsky-Tam}
	\text{FRB:}
	\left\lfloor
		\begin{array}{l}
			z_{k+1} = P_{C} \left[ z_{k} - 2 \gamma F(z_{k}) + \gamma F(z_{k-1}) \right]
		\end{array}
	\right.
\end{equation}
and converges to a solution of \eqref{intro:pb:vi} for $ 0 < \gamma < \nicefrac{1}{2L} $.
Note that FRB can be deducted from FBF by reusing $ F(w_{k-1}) $ instead of $ F(z_{k}) $ in the first line of~\eqref{algo:Tseng}, similarly to how Popov's method can be obtained from EG. Hence~\eqref{algo:Malitsky-Tam} coincides with~\eqref{algo:OGDA} and OGDA in the unconstrained case.

\paragraph{Projected Reflected Gradient (RG) method}
Before investigating FRB, Malitsky~\yrcite{yura-reflection} introduced another similar method where the order of the reflection and the forward step is reversed. In particular, a second forward step can be avoided by evaluating $ F $ at an appropriate linear combination of the iterates. This gives rise to the following method for $ k \geq 1 $
\begin{equation}\label{algo:Malitsky}
	\text{RG:}
	\left\lfloor
		\begin{array}{l}
			w_{k} = 2 z_{k} - z_{k-1} \\
			z_{k+1} = P_{C} \left[ z_{k} - \gamma F(w_{k}) \right]
		\end{array}
	\right.
\end{equation}
which converges to a solution of \eqref{intro:pb:vi} for $ 0 < \gamma < \nicefrac{(\sqrt{2}-1)}{L} $.

Despite the similarities in the construction and iterate convergence, nonasymptotic convergence is less understood in the case of single-call projection methods.
For example Banert and Bo{\c t}~\yrcite{banert2018forward} derived for Tseng's method an ergodic $\mathcal{O}(\nicefrac{1}{k})$ rate in terms of function values in the context of convex optimisation; see also \cite{tseng-minimax} for an ergodic $\mathcal{O}(\nicefrac{1}{\sqrt{k}})$ convergence result in terms of the restricted gap function in the stochastic setting. For Malitsky's RG algorithm convergence in terms of the gap function and residuals like $\mathcal{O}(\nicefrac{1}{\sqrt{k}})$ for the last iterate was established recently~\cite{cai2022acceleratedsingle}.

\subsubsection{Accelerated Solution Methods}\label{sub:accmeth}

\paragraph{Extra Anchored Gradient (EAG) algorithm}
An accelerated algorithm for solving the monotone equation \eqref{intro:pb:eq} that is based on EG, called Extra Anchored Gradient (EAG) algorithm, was proposed by Yoon and Ryu~\yrcite{yoon2021accelerated}. It is designed by using \emph{anchoring}, a technique that can be traced back to Halpern's algorithm~\yrcite{halpern1967fixed}.
This iterative scheme exhibits a convergence rate of
\begin{equation}
	\norm{F(z_{k})} = \mathcal{O} \left( \frac{1}{k} \right)
	\quad \textrm{as } k \to + \infty.
\end{equation}
These considerations have been followed by extension of EAG to the constrained setting~\cite{cai2022acceleratedalgorithms}, where the authors consider for $ k \geq 0 $
\begin{equation}\label{algo:EAG}
	\text{EAG:}
	\left\lfloor
		\begin{array}{l}
			w_{k} = P_{C} \left[ z_{k} - \gamma F(z_{k}) + \frac{1}{k+1} (z_{0} - z_{k}) \right] \\
			z_{k+1} = P_{C} \left[ z_{k} - \gamma F(w_{k}) + \frac{1}{k+1} (z_{0} - z_{k}) \right]
		\end{array}
	\right.
\end{equation}
with $ 0 < \gamma < \nicefrac{1}{\sqrt{3} L} $.
One can notice that the algorithm uses two operator evaluations and two projection steps, like EG, and in the unconstrained case it coincides with its projection free counterpart~\cite{yoon2021accelerated}, maintaining the $\mathcal{O}(\nicefrac{1}{k})$ convergence rate for gap function and residuals.

\paragraph{Accelerated Reflected Gradient (ARG) algorithm}
Similar to the nonaccelerated methods from the previous subsection, one can also reduce the number of necessary operator evaluations and projections to one each per iteration. This was done by investigating an accelerated version~\cite{cai2022acceleratedsingle} of the projected reflected gradient method~\eqref{algo:Malitsky} with convergence rate $\mathcal{O}(\nicefrac{1}{k})$.
The method is given for $ k \geq 1 $ by
\begin{equation}\label{algo:ARG}
	\text{ARG:}
	\left\lfloor
		\begin{array}{l}
			w_{k} = 2 z_{k} - z_{k-1} + \frac{1}{k+1} (z_{0} - z_{k}) \\
				\hspace{9mm}- \frac{1}{k} (z_{0} - z_{k-1}) \\
			z_{k+1} = P_{C} \left[ z_{k} - \gamma F(w_{k}) + \frac{1}{k+1} (z_{0} - z_{k}) \right]
		\end{array}
	\right.
\end{equation}
with $ 0 < \gamma \leq \nicefrac{1}{12 L} $.

It is worth mentioning that for constrained EAG~\cite{cai2022acceleratedalgorithms} and ARG~\cite{cai2022acceleratedsingle} there are no guarantees for the iterates to converge to a solution and that, in spite of the fast theoretical convergence rate, the effect of the \emph{anchor} $z_0$, to which the algorithm returns in every iteration, on the convergence speed is a slowing one, as we will see in the numerical experiments.

\paragraph{Further accelerated algorithms}
Further variants of anchoring based algorithms have been proposed by Tran-Dinh~\yrcite{tran2022connection} and together with Luo~\cite{tran2021halpern}, which all exhibit the same convergence rate in terms of the operator norm as EAG for \eqref{intro:pb:eq}.
Monotone inclusions are also considered in these works, with a more general operator than the normal cone, but this requires either taking a backward step or additionally asking for cocoercivity of $F$. In the same spirit, an $o(\nicefrac{1}{k})$ rate of convergence together with convergence of iterates was shown for an accelerated version of the Krasnosel'ski\u{\i}-Mann algorithm~\cite{bot2022fastKM}.

\paragraph{Explicit Fast OGDA (fOGDA) algorithm}
Another approach which is different from the Halpern-type methods mentioned above was investigated in~\cite{bot2022fastOGDA}.
An appropriate (explicit) discretisation of a second-order dynamical system with vanishing damping term gives rise to an accelerated algorithm related to OGDA, called \emph{fast OGDA} (fOGDA), which will constitute a starting point for our considerations in the following.
The fast OGDA algorithm~\cite{bot2022fastOGDA} for solving the monotone equation~\eqref{intro:pb:eq}  is given for $ k \geq 1 $ by
\begin{equation}
	\text{fOGDA:}
	\left\lfloor
		\begin{array}{l}
			\!\!w_{k} = z_{k} + \frac{k}{k + \alpha} \left( z_{k} - z_{k-1} \right) - \gamma \frac{\alpha}{k + \alpha} F(w_{k-1})\\[1ex]
			\!\!z_{k+1} = w_{k} - \gamma
			\frac{2k + \alpha}{k + \alpha}
			\left( F(w_{k}) - F(w_{k-1}) \right) \\[1ex]
		\end{array}
	\right.
\end{equation}
and converges to a solution of~\eqref{intro:pb:eq}  for $ 0 < \gamma < \nicefrac{1}{4 L} $ and $ \alpha > 2 $. It was shown that fOGDA exhibits convergence rates like
\begin{equation}
	\Gap(z_{k}) = o \left( \frac{1}{k} \right)
	\hspace{0.5em} \text{and} \hspace{0.5em}
	\norm{F(z_{k})} = o \left( \frac{1}{k} \right)
\end{equation}
as $ k \to +\infty $.

\begin{algorithm*}[tb]
	\caption{fOGDA-VI}
	\label{algo:split}
	\begin{algorithmic}
		\STATE {\bfseries Input:} momentum parameter $ \alpha > 2 $; starting values $ z_{0}, \, w_{0} \in \R^{d} $, $ z_{1} \in C $, $ \zeta_{1} \in N_{C}(z_{1}) $; step size $ 0 < \gamma < \nicefrac{1}{4L} $; number of iterations $ K > 1 $.
		\FOR{$k=1$ {\bfseries to} $K$}
		\STATE Compute
		\begin{align}
		w_{k} &= z_{k} + \frac{k}{k + \alpha} \left( z_{k} - z_{k-1} \right) - \gamma \frac{\alpha}{k + \alpha} \left( F(w_{k-1}) + \zeta_{k} \right) \\
		z_{k+1} &= P_{C} \left[ w_{k} - \gamma \left( 1 + \frac{k}{k + \alpha} \right) \left( F(w_{k}) - F(w_{k-1}) - \zeta_{k} \right) \right] \\
		\zeta_{k+1} &= \frac{k + \alpha}{\gamma (2 k + \alpha)} \left( w_{k} - z_{k+1} \right) - \left( F(w_{k}) - F(w_{k-1}) - \zeta_{k} \right)
		\end{align}
		\ENDFOR
	\end{algorithmic}
\end{algorithm*}

\section{Main Results}\label{sec:main}
In this section, we will first motivate the changes necessary to extend fOGDA to solve the variational inequality~\eqref{intro:pb:vi} and introduce our newly proposed method which we call \emph{fOGDA-VI}. This is followed by formally stating the convergence results of fOGDA-VI -- convergence of the iterates to a solution as well as convergence in terms of the restricted gap and the residuals like $ o(\nicefrac{1}{k}) $ for the \emph{last} iterate.

\subsection{Extending fOGDA to the Constrained Case}
It can be seen empirically, that incorporating one (or more) projections to the unconstrained fOGDA method in a naive way, similar to FRB or Popov, is not sufficient to obtain a solution method for~\eqref{intro:pb:vi}. Instead, the idea is to introduce for every $ k \geq 1 $ an appropriate element $ \zeta_{k} $ of the normal cone $ N_{C}(z_{k}) $. This is done by replacing $F(w_{k-1})$ by the sum $ F(w_{k-1}) + \zeta_{k} $. One might think that because of this the suitable choice would be to take $ \zeta_{k} \in N_{C}(w_{k-1}) $, but this is not the case which can be motivated as follows.
For the unconstrained case, i.e., when solving the monotone equation~\eqref{intro:pb:eq}, we want to find a sequence $ (z_{k})_{k \geq 1} $ yielding $ \norm{F(z_{k})} \to 0 $. However, in the constrained case, i.e., when tackling the monotone inclusion~\eqref{intro:pb:mi}, we aim to establish sequences $ (z_{k})_{k \geq 1} $, $ (\zeta_{k})_{k \geq 1} $ with $ \zeta_{k} \in N_{C}(z_{k}) $ such that
\begin{equation}
	\norm{F(z_{k}) + \zeta_{k}} \to 0.
\end{equation}
Then instead of fOGDA we have an algorithm which is given for every $k \geq 1$ by
\begin{equation}\label{eq:fogda-implicit-proj}
\begin{split}
	w_{k} &= z_{k} + \tfrac{k}{k + \alpha} \left( z_{k} - z_{k-1} \right) - \gamma \tfrac{\alpha}{k + \alpha} \left( F(w_{k-1}) + \zeta_{k} \right), \\
	z_{k+1}	&= w_{k} - \gamma \tfrac{2k + \alpha}{k + \alpha} \left( F(w_{k}) - F(w_{k-1}) + \zeta_{k+1} - \zeta_{k} \right).
\end{split}
\end{equation}
At first glance this method seems to be implicit, as both $ z_{k+1} $ and $ \zeta_{k+1} $ appear on the same line. However, the second line in~\eqref{eq:fogda-implicit-proj} can be used to formulate a projection step (see Appendix~\ref{app:prelim} for details). Even though from the perspective of~\eqref{eq:projNC} the appearance of the normal cone element is a natural consequence of the projection, finding its correct formulation is a highly non-trivial task.
These considerations give rise to Algorithm~\ref{algo:split}, which we call \emph{fOGDA-VI}.

The two main differences of our proposed method -- introduction of a projection at a specific spot as well as explicit computation of a particular element in the normal cone -- are deemed to be necessary. Changing (or even neglecting) either of them results in algorithms that fail to converge in general.

\begin{remark}
	Initialisation of fOGDA-VI, however, is easy as for general $ z_{1} \in C $ it is sufficient to take $ \zeta_{1} = 0 $. With arbitrary $ \hat{z} \in \R^{d} $, one can also take $ z_{1} := P_{C}(\hat{z}) $ and $ \zeta_{1} := \hat{z} - z_{1} \in N_{C}(z_{1}) $
\end{remark}

\subsection{Convergence Statements}
The first main result concerns the convergence of the sequence of iterates to an element in $\Omega$.
\begin{theorem}\label{thm:conv}
	Let $(z_{k})_{k \geq 0}$ be the sequence generated by Algorithm~\ref{algo:split}. Then the sequence $(z_{k})_{k \geq 0}$ converges to a solution of \eqref{intro:pb:vi}.
\end{theorem}

The central idea for the proof is to define an appropriate family of energy functions $(\cG_{\lambda, k})_{k \geq 0}$, where $ \lambda \geq 0 $ is a parameter that depends on $ \alpha $, which dissipate over the course of the algorithm to obtain convergence or summability of various helpful quantities.
Even though we are not able to enforce the family of discrete energies $(\cG_{\lambda, k})_{k \geq 0}$ to have an actual nonincreasing property, we can at least show that for every $k \geq 0$
\begin{equation}\label{inq:G}
	\cG_{\lambda, k+1} \leq (1 + d_{\lambda, k}) \cG_{\lambda, k} - b_{\lambda, k},
\end{equation}
with some sequences $(b_{\lambda, k})_{k \geq 0}$ and $(d_{\lambda, k})_{k \geq 0}$. The aim is to control these two sequences in such a way that we can still derive some beneficial asymptotic results for $(\cG_{\lambda, k}) _{k \geq 0}$. As the additional terms are not necessarily nonnegative, a novel Lyapunov analysis is needed.
For instance, we show that there exist $0 \leq \underline{\lambda} \left( \alpha \right) < \overline{\lambda} \left( \alpha \right) \leq \nicefrac{(3 \alpha - 2)}{4}$ such that every $\underline{\lambda} \left( \alpha \right) < \lambda < \overline{\lambda} \left( \alpha \right)$ provides an energy function $(\cG_{\lambda, k}) _{k \geq 0}$ that is bounded from below and nonnegative sequences $(b_{\lambda, k})_{k \geq 0}$ and $(d_{\lambda, k})_{k \geq 0}$ with  $\sum_{k \geq 0} d_{\lambda, k} < +\infty$ such that the inequality \eqref{inq:G} holds for $k$ large enough. This allows us to conclude that $\lim_{k \to + \infty} \cG_{\lambda,k} \in \R$ exists, see Lemma~\ref{lem:quasi-Fej} for more details.

From this we can then verify that the first condition of Opial's lemma, see Lemma~\ref{lem:opial}, is fulfilled, while its second condition follows from the maximal monotonicity of $F + N_{C}$, see Proposition~\ref{prop:seq-closed}.

The asymptotic convergence of the iterates is complemented by statements about convergence rates in terms of the restricted gap as well as the natural residual for the last iterate.
\begin{theorem}\label{thm:gap-res}
	Let $z^{\ast} \in \Omega$ be a solution of~\eqref{intro:pb:vi} and let $(z_{k})_{k \geq 0}$ be the sequence generated by Algorithm~\ref{algo:split}. Then, as $ k \to + \infty $, we have
	\begin{equation}
		\Gap(z_{k}) = o \left( \frac{1}{k} \right)
		\quad \textrm{and} \quad
		\Res(z_{k}) = o \left( \frac{1}{k} \right).
	\end{equation}
\end{theorem}

\begin{remark}
	The tangent residual exhibits the same last iterate convergence rate as the natural residual, i.e.,
	\begin{equation}
		r(z_{k}) = o \left( \frac{1}{k} \right).
	\end{equation}
	In fact, we use the observation $ \Res(z) \leq \norm{F(z) + \zeta} $, see~\eqref{eq:gaps}, in the proof of Theorem~\ref{thm:gap-res} to obtain the convergence rate for the natural residual. As the restricted gap and the natural residual are mostly used in the literature (and are probably more intuitive) to quantify the convergence behaviour of numerical methods for variational inequalities, we opted to present Theorem~\ref{thm:gap-res} in the above manner.
\end{remark}

\section{Numerical Experiments}\label{sec:num}
In this section we provide two numerical experiments to complement our theoretical results. For the first one we treat a two-player zero sum game, which amounts to solving a bilinear saddle point problem constrained by standard simplexes. The second one consists of application of fOGDA-VI to the training of GANs.

\subsection{Two-player Zero Sum Game}\label{sub:bilin}

\begin{figure}[ht]
	\vskip 0.2in
	\begin{center}
		\centerline{\includegraphics[width=\columnwidth]{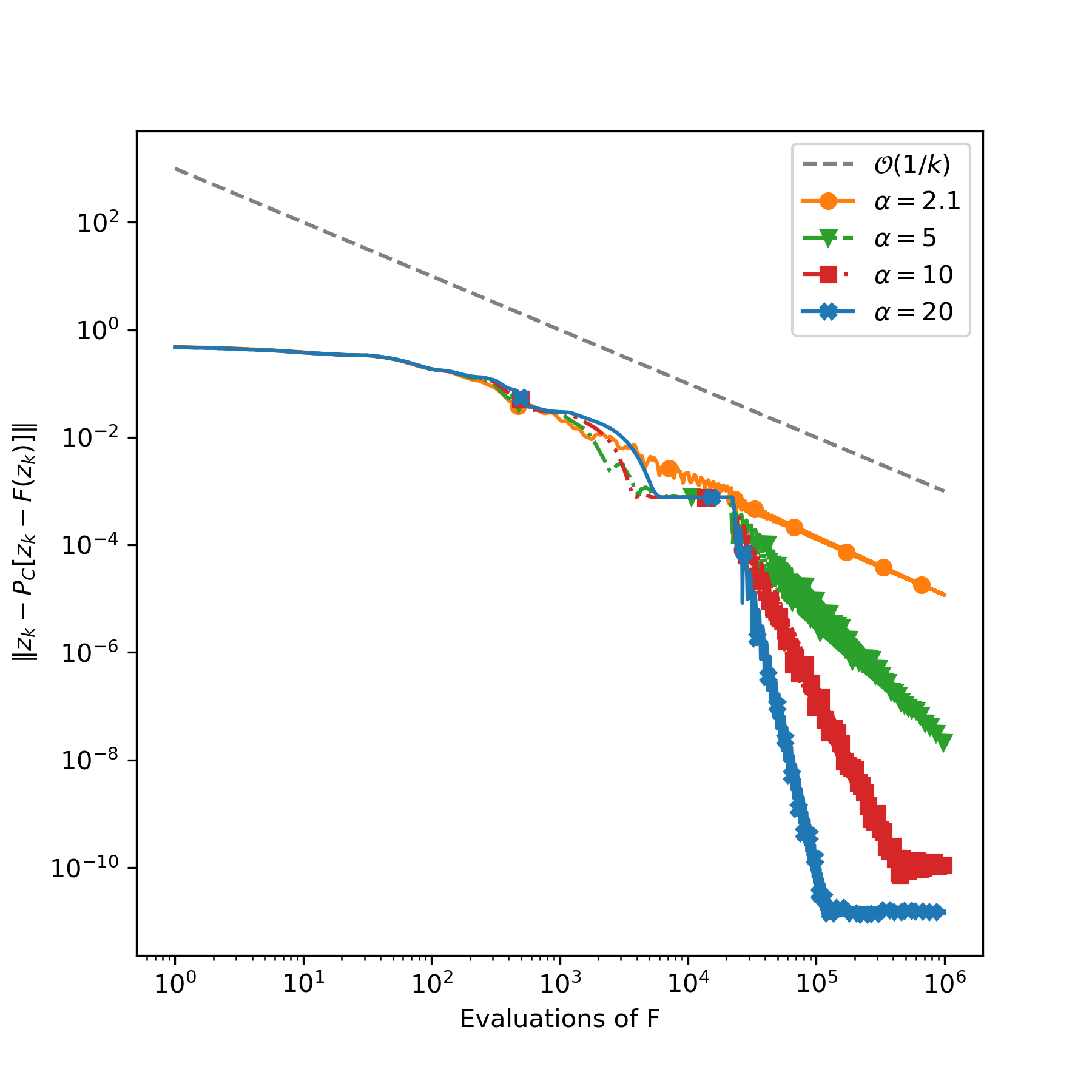}}
		\caption{Comparison of different momentum parameters $ \alpha > 2 $ in Algorithm~\ref{algo:split} in terms of the natural residual.}
		\label{fig:fogda-alpha}
	\end{center}
	\vskip -0.2in
\end{figure}

We aim to solve a two-player zero sum game with mixed strategies, which means that we need to solve the following bilinear saddle point problem,
\begin{equation}\label{eq:zero-sum}
	\min_{x \in \Delta^{m}} \, \max_{y \in \Delta^{n}} \; \Phi(x, y) := x^{T} A y,
\end{equation}
where $ A \in \R^{m \times n} $ is a given pay-off matrix and $ \Delta^{d} = \{ v \in \R^{d}_{+} \mid \sum_{i=1}^{d} v_{i} = 1 \} $ denotes the $ d $-dimensional standard simplex. Recall that a solution of~\eqref{eq:zero-sum} is given by a saddle point $ (x^{\ast}, y^{\ast}) \in \Delta^{m} \times \Delta^{n} $ satisfying
\begin{equation}
	\Phi(x^{\ast}, y) \leq \Phi(x^{\ast}, y^{\ast}) \leq \Phi(x, y^{\ast})
	\quad \forall (x,y) \in \Delta^{m} \times \Delta^{n}.
\end{equation}
This leads to a monotone inclusion problem~\eqref{intro:pb:mi} with $ C = \Delta^{m} \times \Delta^{n} $ and
\begin{equation}
	F: \R^{m} \times \R^{n} \to \R^{m} \times \R^{n},
	\hspace{0.5em}
	\begin{pmatrix}
		x\\
		y
	\end{pmatrix}
	\mapsto
	\begin{pmatrix}
		\nx \Phi (x, y)\\
		\minus \ny \Phi (x, y)\\
	\end{pmatrix},
\end{equation}
which gives
\begin{equation}
	F(x,y)
	=
	\begin{pmatrix}
		Ay\\
		\minus A^{T}x
	\end{pmatrix}
	=
	\begin{pmatrix}
		0 & A\\
		-A^{T} & 0
	\end{pmatrix}
	\begin{pmatrix}
		x\\
		y
	\end{pmatrix}.
\end{equation}

\begin{figure}[ht]
	\vskip 0.2in
	\begin{center}
		\centerline{\includegraphics[width=\columnwidth]{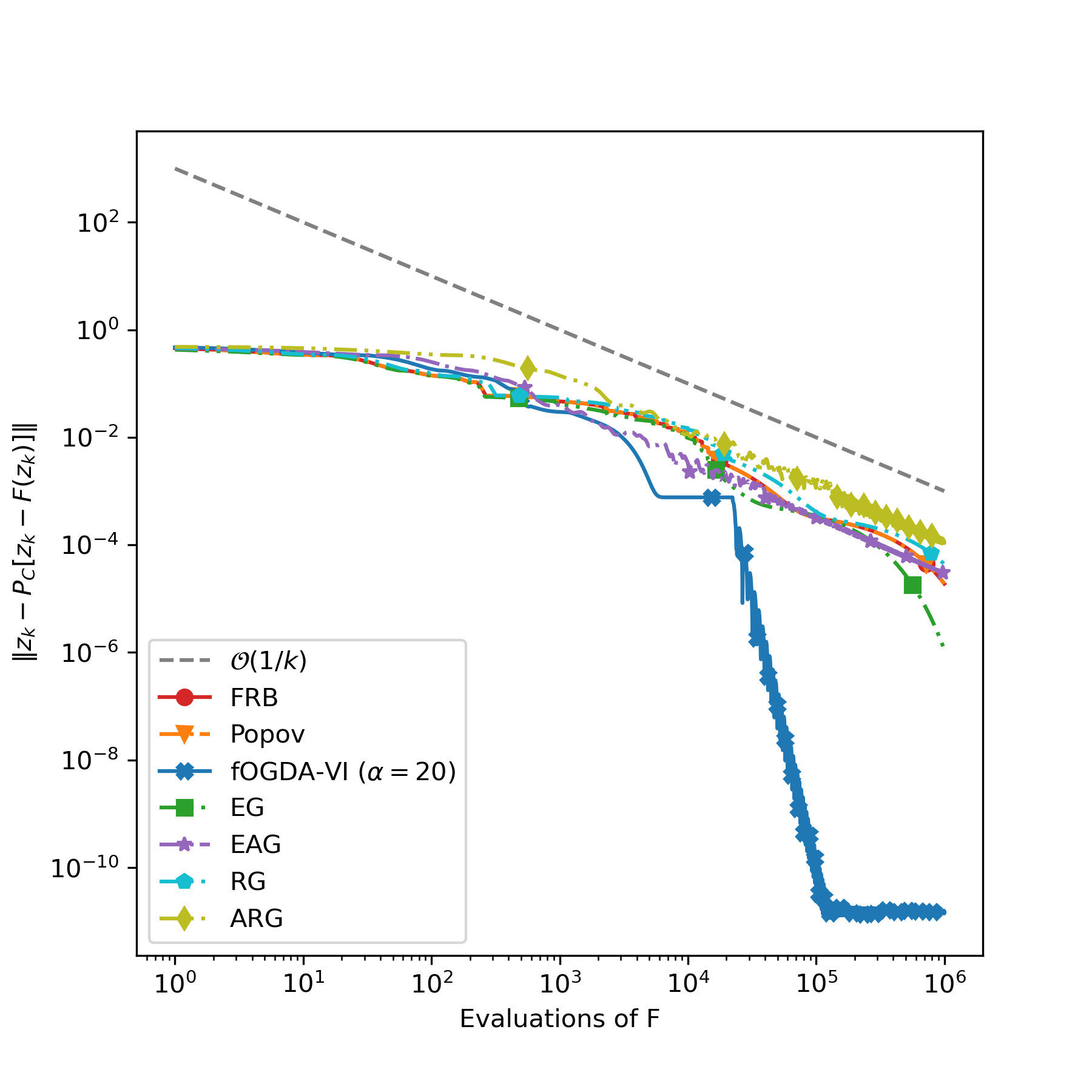}}
		\caption{Comparison of different methods in terms of the natural residual.}
		\label{fig:fogda-res}
	\end{center}
	\vskip -0.2in
\end{figure}

Notice that $ F $ is Lipschitz continuous but not cocoercive, thus indeed the regular \emph{Projected Gradient Descent Ascent} algorithm (PGDA), which is in this case nothing else than the FB algorithm, cannot be applied.

\begin{table*}[t]
	\caption{Comparison of fOGDA-VI with LA-GDA in terms of Fr\'{e}chet Inception Distance (FID; lower is better) and Inception Score (IS; higher is better). We report the best obtained scores, averaged over 5 runs with 500,000 iterations each. For all considered methods we evaluated the last (non averaged) iterates, the uniform average, the exponential moving average (EMA) and the EMA on the ``slow weights'' for the method incorporating ``lookahead''. Best scores for each metric are in boldface.}
	\label{tab:all-scores}
	\vskip 0.15in
	\begin{center}
		\begin{small}
			\begin{adjustbox}{max width=\textwidth}
				\begin{tabular}{lcccccccc}
					\toprule
					& \multicolumn{4}{c}{FID} & \multicolumn{4}{c}{IS} \\
					\cmidrule(lr){2-5} \cmidrule(lr){6-9}
					Method & non avg. & uniform avg. & EMA & EMA-slow & non avg. & uniform avg. & EMA & EMA-slow \\
					\midrule
					fOGDA & 18.49 $\pm$ 1.09 & 17.38 $\pm$ 1.69 & 18.51 $\pm$ 1.13 & -- & 7.82 $\pm$ .07 & 8.7 $\pm$ .15 & 8.1 $\pm$ .15 & -- \\
					LA-GDA & 16.7 $\pm$ .67 & 16.02 $\pm$ .84 & 16.84 $\pm$ .71 & \textbf{15.31} $\pm$ 1.27 & 7.88 $\pm$ .08 & \textbf{8.76} $\pm$ .19 & 8.29 $\pm$ .07 & 8.59 $\pm$ .1 \\
					\bottomrule
				\end{tabular}
			\end{adjustbox}
		\end{small}
	\end{center}
	\vskip -0.1in
\end{table*}

For our experiments we choose $ d = n = 50 $ and $ A $ to have entries drawn from the uniform distribution on the half-open interval $ \left[ 0, 1 \right) $.
As the parameter $ \alpha > 2 $ can be chosen arbitrarily, we do a comparison of different values in terms of the natural residual in Figure~\ref{fig:fogda-alpha} to gain more insight. Note that there is no upper bound for $ \alpha $ that would be given on the basis of the theoretical considerations. As it turns out, from a certain point on after a period where all values of $ \alpha $ perform similarly, bigger choices for $ \alpha $ seem to give better results with faster convergence (even though the convergence rate of $ o(\nicefrac{1}{k}) $ is the same for all choices). When going to ``extremely big'' choices of $ \alpha $ we could not only observe further boost in convergence speed, but also increased oscillatory behaviour after a certain point. Whether fOGDA-VI for $ \alpha \to +\infty $ amounts to a convergent method is not obvious; for the unconstrained fOGDA by \cite{bot2022fastOGDA} one can see that this would lead to the unaccelerated OGDA method.

Concluding, we show a comparison of different methods that we have encountered in Sections~\ref{sub:nonaccmeth} and~\ref{sub:accmeth} in Figure~\ref{fig:fogda-res} where we report results on the natural residual. We see that fOGDA-VI clearly outperforms all other methods while only requiring one evaluation of $ F $ and one projection in each iteration.

\subsection{GAN Training}\label{sub:fOGDA-GAN}
Generative Adversarial Networks (GANs)~\cite{GAN} form a powerful class of generative models that can produce for example unseen realistic images. Originally the problem was posed as a zero sum game between two adversarial players played by two neural networks, called \emph{generator} and \emph{discriminator}, that try to minimise and maximise the same loss function, respectively.
The minimax structure of the underlying optimisation problem generally leads to cycling behaviour during the training process, making GANs notoriously hard to optimise~\cite{numerics-of-gans,mesched-GAN-converge}.
As it was shown empirically that principled methods that are used to solve variational inequalities~\cite{VIP-GAN} and monotone inclusions~\cite{tseng-minimax} can prove beneficial in the training process, we will apply our proposed algorithm fOGDA-VI to train ResNet architectures on the CIFAR-10 dataset.

For our experiments we use ResNet~\cite{resnet} architectures, see Appendix~\ref{app:resnet}, with the hinge version of the adversarial non-saturating loss~\cite{SN-GAN} trained on the CIFAR-10~\cite{CIFAR10} data set, which consists of 60,000 $ (32 \times 32 \times 3) $-images in 10 classes, with 6,000 images per class. The metrics used to evaluate the generated images are the inception score~\cite{is-original} (IS; higher is better) and the Fr{\' e}chet inception distance~\cite{FID} (FID; lower is better), both computed on 50,000 samples in their original implementations.

Furthermore, in our experiments instead of stochastic gradients we use the Adam optimiser~\cite{adam} with parameters $ \beta_{1} = 0 $ and $ \beta_{2} = 0.9 $ that were used in recent experiments~\cite{lookahead-minmax} outperforming the class-dependent BigGAN~\cite{BigGAN} model on CIFAR-10. Additionally, we keep the batch size and the ratio of discriminator and generator updates the same as in~\cite{lookahead-minmax}.

Since we have mini batch updates for the GAN experiments instead of taking the full gradient we perform significantly more steps to incorporate the entire gradient information. Because of this the iterator $ k $ in Algorithm~\ref{algo:split} might grow too large soon, so we conducted experiments to update the iterator $ k $ only every $ n $-th step with different choices of $ n $. Furthermore, we also did a hyperparameter search regarding the learning rate and the momentum parameter $ \alpha > 2 $ in Algorithm~\ref{algo:split}.

\begin{table}[hb]
	\caption{Overall best obtained scores in terms of Fr\'{e}chet Inception Distance (FID; lower is better) and Inception Score (IS; higher is better) for the last (non averaged) iterates. Best scores for each metric are in boldface.}
	\label{tab:best-scores}
	\vskip 0.15in
	\begin{center}
		\begin{small}
			\begin{adjustbox}{max width=\textwidth}
				\begin{tabular}{lcc}
					\toprule
					Method & FID & IS \\
					\midrule
					fOGDA-VI & 15.69 & 8.91 \\
					LA-GDA & \textbf{14.09} & \textbf{9.06} \\
					\bottomrule
				\end{tabular}
			\end{adjustbox}
		\end{small}
	\end{center}
	\vskip -0.1in
\end{table}

\begin{figure*}[bt]%
	\vskip 0.2in
	\begin{center}
		\subfigure[FID.]{\includegraphics[width=\columnwidth]{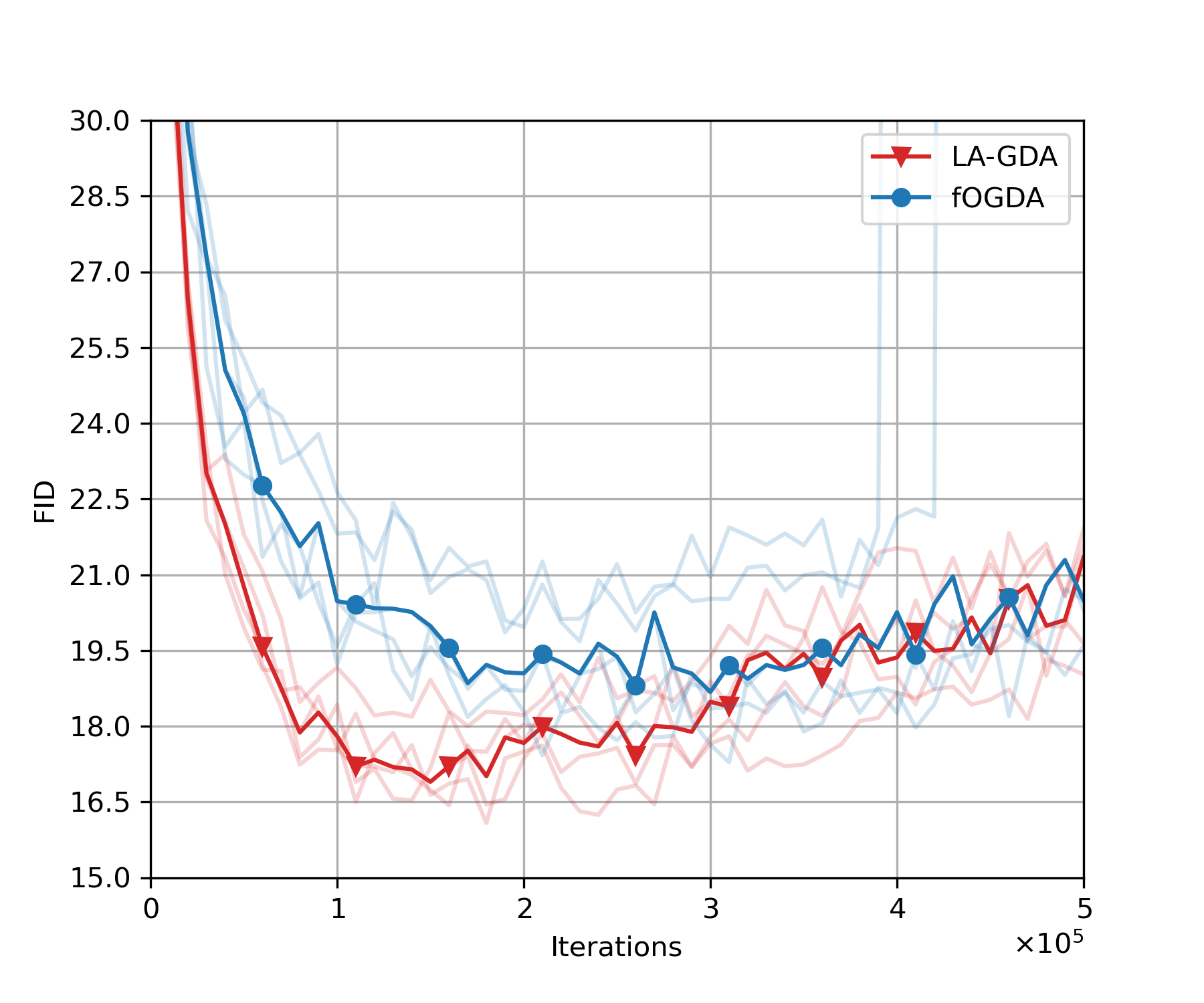}}
		\subfigure[IS.]{\includegraphics[width=\columnwidth]{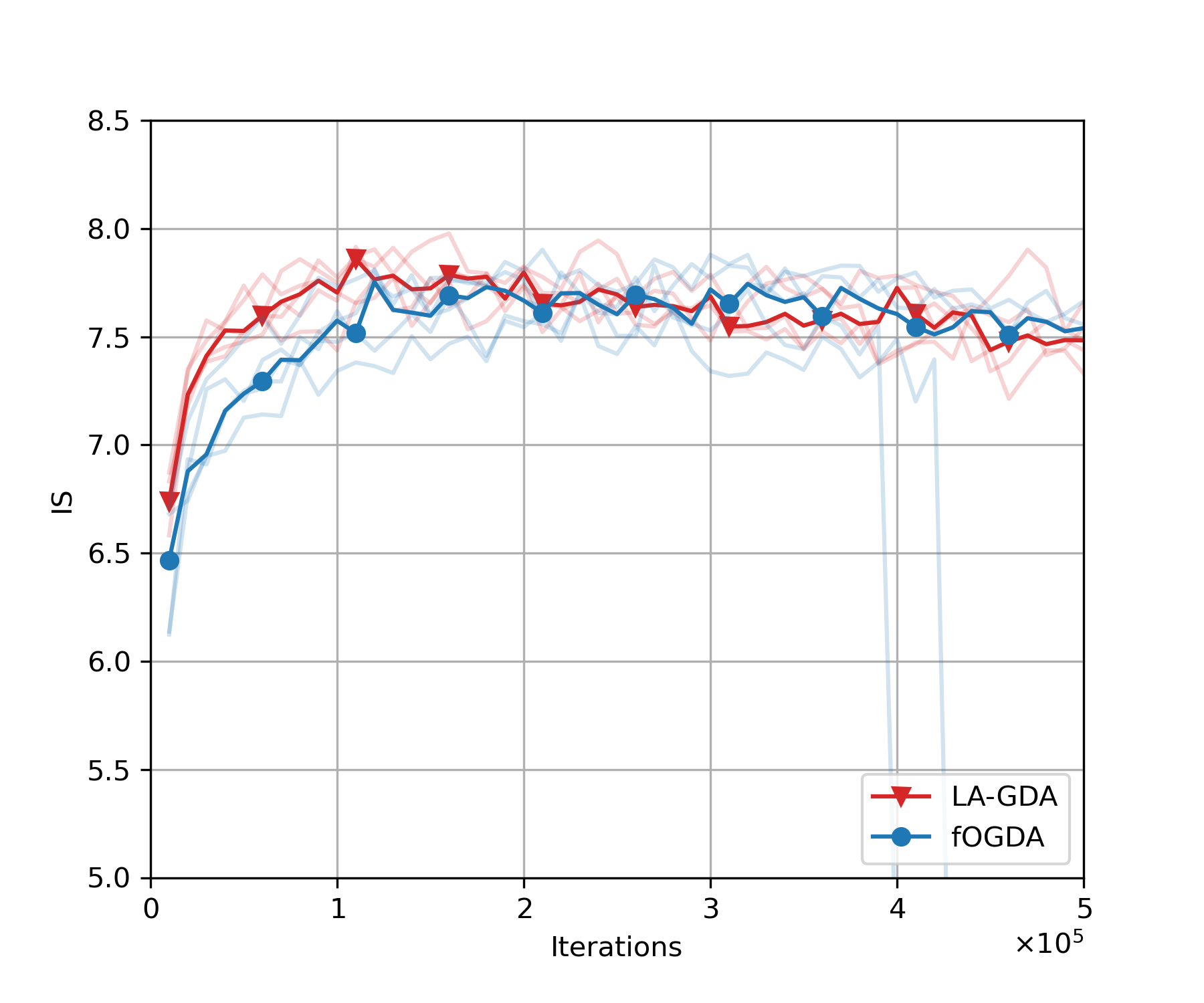}}
		\caption{The median and the individual runs are illustrated with ticker solid lines and transparent lines, respectively. We report (a) FID and (b) IS for the last (non averaged) iterates.}
		\label{fig:all-runs}
	\end{center}
	\vskip -0.2in
\end{figure*}

When performing the hyperparameter search for the momentum parameter $ \alpha $, we experienced that, just as in the theoretically justified setting of Section~\ref{sub:bilin}, bigger values seemed to perform better. Regarding the frequency of iterator updates we also observed better behaviour for bigger values of $ n $ in general. The parameters we used for the fOGDA-VI experiments were $ \alpha = 100 $ and $ n = 1000 $, and a learning rate of $ \gamma = 0.0001 $.
We compared the results obtained by fOGDA-VI with the best method from~\cite{lookahead-minmax}, a variant of Gradient Descent Ascent incorporating averaging during the training which they call ``lookahead'', resulting in a method we denote by LA-GDA, for the convergence properties of which no theoretical evidence is available. For the LA-GDA experiments we kept all hyperparameters as reported by Chavdarova et al.~\yrcite{lookahead-minmax}.  For both methods we conducted 5 runs with 500,000 iterations each.

To further reduce the cycling characteristics of GAN training two commonly used techniques in practice are the (uniform) averaging and the exponential moving averaging (EMA) of the network weights. The beneficial effects of uniform and exponential averaging can also be observed in our experiments (see Table~\ref{tab:all-scores}), however uniform averaging seems to have a stronger impact on the scores.

In Table~\ref{tab:all-scores} we report a comparison of fOGDA-VI with LA-GDA in terms of Fr\'{e}chet Inception Distance (FID; lower is better)~\cite{FID} and Inception Score (IS; higher is better)~\cite{is-original}. We report the best obtained scores, averaged over 5 runs with 500,000 iterations each. For both methods we evaluated the last (non averaged) iterates, the uniform average, the exponential moving average (EMA) and the EMA on the ``slow weights'' for LA-GDA. One can see that while both considered methods give comparable results and show similar behaviour, the best score for both FID and IS is obtained by LA-GDA. An interesting observation is that the FID scores obtained by LA-GDA seem to be significantly worse than those reported in~\cite{lookahead-minmax}, while it exhibits higher reported IS values.

The scores reported in Table~\ref{tab:best-scores}, where we list the overall best obtained values for both metrics, support the observations from Table~\ref{tab:all-scores}. In general, the results for fOGDA-VI and LA-GDA are on a similar level with the latter giving the altogether best scores.

Figure~\ref{fig:all-runs} shows all five individual runs and the respective median for both methods in terms of (a) FID and (b) IS. It can be observed that LA-GDA achieves better results than fOGDA-VI during the first 200,000 iterations, while from then on the both methods achieve similar scores. As it appears, the medians of fOGDA-VI seem to stay more consistently on the level of the optimal scores while LA-GDA worsens again over time.

\section{Conclusion}
In this work we proposed a novel algorithm, called fOGDA-VI, to solve monotone variational inequalities that recovers the explicit fOGDA method from~\cite{bot2022fastOGDA} in the unconstrained case. We showed that fOGDA-VI exhibits a better rate of convergence than other accelerated methods, giving a rate of convergence like $ o(\nicefrac{1}{k}) $ in terms of the restricted gap function and the tangent and natural residuals, while still maintaining convergence of the iterates to a solution of the variational inequality under investigation.
To validate our method in practice we treated a constrained bilinear minimax problem for which we obtained superior behaviour on this theoretically justified task. Moreover, application of fOGDA-VI to the training of GANs gives promising results even in practical settings that do not warrant the required assumptions.

\section*{Acknowledgements}
The authors would like to thank the anonymous reviewers and the program chairs for their valuable suggestions and comments that have improved the quality of the paper.
Michael Sedlmayer would like to acknowledge support from the Austrian Research Promotion Agency (FFG), project ``Smart operation of wind turbines under icing conditions (SOWINDIC)''.
Dang-Khoa Nguyen would like to acknowledge support from the Austrian Science Fund (FWF), project P~34922.
Radu Ioan Bo{\c t} would like to acknowledge partial support from the Austrian Science Fund (FWF), projects W~1260 and P~34922.

\bibliography{fOGDA-VI}
\bibliographystyle{icml2023}

\newpage
\appendix
\onecolumn

\section{Auxiliary Results}
In the following we present auxiliary results that will be necessary for the convergence analysis in Appendix~\ref{app:proofs}.

We start this section with a result which characterises the convergence of sequences.
\begin{lemma}[{see Lemma~32 in \cite{bot2022fastOGDA}}]\label{lem:lim-u-k}
	Let $a \geq 1$ and $\left( q_{k} \right)_{k \geq 1}$ be a bounded sequence in $\R^{d}$ such that
	\begin{equation}
	\lim_{k \to + \infty} \left( q_{k+1} + \frac{k}{a} \left( q_{k+1} - q_{k} \right) \right) = p \in \R^{d} .
	\end{equation}
	Then $\lim_{k \to + \infty} q_{k} = p$.
\end{lemma}

The following result is a particular instance of Lemma~5.31 in \cite{BauschkeCombettes2}.
\begin{lemma}\label{lem:quasi-Fej}
	Let $ (a_{k})_{k \geq 1}$, $ (b_{k})_{k \geq 1}$, $ (d_{k})_{k \geq 1} $ and $ (d_{k})_{k \geq 1} $ be sequences of real numbers. Assume that $ (a_{k})_{k \geq 1} $ is bounded from below, $ (b_{k})_{k \geq 1} $ and $ (d_{k})_{k \geq 1} $ are nonnegative sequences such that $ \sum_{k \geq 1} d_{k} < + \infty$. If
	\begin{equation}\label{quasi-Fej:inq}
	a_{k+1} \leq \left( 1 + d_{k} \right) a_{k} - b_{k}
	\quad \forall k \geq 1,
	\end{equation}
	then the following statements are true:
	\begin{enumerate}[(i)]
		\item the sequence $ (b_{k})_{k \geq 1} $ is summable, i.e., $ \sum_{k \geq 1} b_{k} < +\infty$;

		\item the sequence $ (a_{k})_{k \geq 1} $ is convergent.
	\end{enumerate}
\end{lemma}

To show convergence of the sequence of iterates we will use the following result, which is a finite dimensional version of the so-called \emph{Opial Lemma} \cite{Opial}.
\begin{lemma}\label{lem:opial}
	Let $ S \subseteq \R^{d} $ be a nonempty set and $ (z_{k})_{k \geq 1} \subseteq \R^{d} $ a sequence such that the following two conditions hold:
	\begin{enumerate}[(i)]
		\item for every $ z^{\ast} \in S $, $ \lim_{k \to + \infty} \norm{z_{k} - z^{\ast}} $ exists;

		\item every cluster point of $ (z_{k})_{k \geq 1} $ belongs to $S$.
	\end{enumerate}
	Then $ (z_{k})_{k \geq 0} $ converges to an element in $S$.
\end{lemma}

The following result about maximal monotone operators will be crucial in the convergence analysis to verify item (ii) from Lemma~\ref{lem:opial}.
\begin{proposition}[{see Proposition~20.33 in \cite{BauschkeCombettes}}]\label{prop:seq-closed}
	Let $ A: \R^{d} \to 2^{\R^{d}} $ be maximal monotone, with $ \gra A $ denoting the graph of $ A $. Then $ \gra A $ is closed, i.e., for every sequence $ (z_{k}, v_{k})_{k \geq 1} $ in $ \gra A $ and every $ (z, v) \in \R^{d} \times \R^{d} $, if $ z_{k} \to z $ and $ v_{k} \to v $, then $ (z, v) \in \gra A $.
\end{proposition}

\begin{lemma}
	\label{lem:quad}
	Let $a, b, c \in \R$ be such that $a \neq 0$ and $b^{2} - ac \leq 0$.
	The following statements are true:
	\begin{enumerate}[(i)]
		\item
		\label{quad:vec-pos}
		if $a > 0$, then
		\begin{equation*}
		a \left\lVert x \right\rVert ^{2} + 2b \left\langle x, y \right\rangle + c \left\lVert y \right\rVert ^{2} \geq 0 \quad \forall x, y \in \R^{d} ;
		\end{equation*}

		\item
		\label{quad:vec}
		if $a < 0$, then
		\begin{equation*}
		a \left\lVert x \right\rVert ^{2} + 2b \left\langle x, y \right\rangle + c \left\lVert y \right\rVert ^{2} \leq 0 \quad \forall x, y \in \R^{d} .
		\end{equation*}
	\end{enumerate}
\end{lemma}

\section{Convergence Analysis}\label{app:proofs}
In the following section we will establish the necessary convergence analysis to prove the main theoretical results of this work Theorem~\ref{thm:conv} and Theorem~\ref{thm:gap-res}.

\subsection{Notation and Preliminary Considerations}\label{app:prelim}
We start with proving the fact, that indeed
$$ \Res(z) \leq r(z) \quad \forall z \in \R^{d}.$$
To this end is sufficient to only consider the case when $ z \in C $ as otherwise the inequality trivially holds. Let $ z \in C $ and $ \zeta \in N_{C}(z) $, then by the following equivalence (see Proposition~6.45 in \cite{BauschkeCombettes})
\begin{equation}\label{eq:projNC}
	p = P_{C}(v)
	\quad \Leftrightarrow \quad
	\exists \zeta \in N_{C}(p) \text{ such that } v - p = \zeta,
\end{equation}
we obtain that $ z = P_{C}[z + \zeta] $. Since the projection onto a nonempty closed convex set is nonexpansive, we then deduce that
\begin{equation}\label{eq:gaps}
	\Res(z)
	= \left\lVert z - P_{C} \left[ z - F(z) \right] \right\rVert
	= \left\lVert P_{C} \left[ z + \zeta \right] - P_{C} \left[ z - F(z) \right] \right\rVert
	\leq \left\lVert F(z) + \zeta \right\rVert .
\end{equation}
Since $\zeta \in N_{C} (z)$ was arbitrary, we conclude that $ \Res(z) \leq r(z) $.

The characterisation~\eqref{eq:projNC} can be also used to deduce a projection step from the second line of~\eqref{eq:fogda-implicit-proj}. For all $ k \geq 1 $ we have
\begin{equation}
	z_{k+1}	= w_{k} - \gamma \left( 1 + \frac{k}{k + \alpha}\right) \left( F(w_{k}) - F(w_{k-1}) + \zeta_{k+1} - \zeta_{k} \right)
\end{equation}
with $ \zeta_{k} \in N_{C}(z_{k}) $, which using~\eqref{eq:projNC} is equivalent to
\begin{equation}
	z_{k+1} = P_{C} \left[ w_{k} - \gamma \left( 1 + \frac{k}{k + \alpha} \right) \left( F(w_{k}) - F(w_{k-1}) - \zeta_{k} \right) \right],
\end{equation}
using that for $ \lambda > 0 $
\begin{equation}
	\zeta \in N_{C}(z)
	\quad \Leftrightarrow \quad
	\lambda \zeta \in N_{C}(z).
\end{equation}

In the following for every $ k \geq 1 $ we will use the notation
\begin{equation}\label{eq:vk}
	v_{k} := F(w_{k-1}) + \zeta_{k},
\end{equation}
which we plug into Algorithm~\ref{algo:split} to obtain
\begin{subequations}
\begin{align}
	w_{k} &= z_{k} + \dfrac{k}{k + \alpha} \Big( z_{k} - z_{k-1} \Big) - \gamma \dfrac{\alpha}{k + \alpha} v_{k} \label{algo:im:bz}, \\
	z_{k+1}	&= w_{k} - \gamma \left( 1 + \dfrac{k}{k + \alpha} \right) \Big( v_{k+1} - v_{k} \Big). \label{algo:im:z}\\
\end{align}
\end{subequations}
Since $ 0 < \gamma < \nicefrac{1}{4L}$ there exists $0 < \varepsilon < 1$ such that
\begin{equation}\label{defi:s-e}
	\gamma = \frac{1 - \varepsilon}{4L}.
\end{equation}
Hence, the definition of $v_{k}$ together with the Lipschitz continuity of $F$ give us
\begin{align}\label{split:Lip}
 \left\lVert \zeta_{k+1} + F \left( z_{k+1} \right) - v_{k+1} \right\rVert
	& = \left\lVert F \left( z_{k+1} \right) - F \left( w_{k} \right) \right\rVert
	\leq L \left\lVert z_{k+1} - w_{k} \right\rVert \\
	&\leq \gamma L \left\lVert v_{k+1} - v_{k} \right\rVert
	\leq \frac{1 - \varepsilon}{4} \left\lVert v_{k+1} - v_{k} \right\rVert
	\leq \dfrac{1}{4} \left\lVert v_{k+1} - v_{k} \right\rVert
	\leq \left\lVert v_{k+1} - v_{k} \right\rVert.
\end{align}
By summing up~\eqref{algo:im:bz} and~\eqref{algo:im:z} we find for every $k \geq 1$
\begin{equation}\label{split:d-u}
	\left( k + \alpha \right) \left( z_{k+1} - z_{k} \right) - k \left( z_{k} - z_{k-1} \right)
	= \minus \alpha \gamma v_{k+1} - 2 \gamma k \left( v_{k+1} - v_{k} \right) .
\end{equation}

Let $0 \leq \lambda \leq \alpha - 1$ and $z^{\ast} \in \Omega$. Following \cite{bot2022fastOGDA} we denote for every $k \geq 1$
\begin{align}
	\label{split:defi:u-k-1-lambda}
	u_{\lambda,k} &:= 2 \lambda \left( z_{k} - z^{\ast} \right) + 2k \left( z_{k} - z_{k-1} \right) + \dfrac{3 \alpha - 2}{\alpha - 1} \gamma k v_{k},\\
	\label{split:defi:E-k}
	\cE_{\lambda,k}	&:= \dfrac{1}{2} \left\lVert u_{\lambda,k} \right\rVert ^{2} + 2 \lambda \left( \alpha - 1 - \lambda \right) \left\lVert z_{k} - z^{\ast} \right\rVert ^{2} + \dfrac{2 \left( \alpha - 2 \right)}{\alpha - 1}  \lambda \gamma k \left\langle z_{k} - z^{\ast}, v_{k} \right\rangle\\
		\nonumber
		&\quad\phantom{:} + \dfrac{\alpha - 2}{\alpha - 1} \gamma^{2} k \left( \dfrac{3 \alpha - 2}{2 \left( \alpha - 1 \right)} k + \alpha \right) \left\lVert v_{k} \right\rVert ^{2},\\
	\label{split:defi:F}
	\cG_{\lambda,k} &:= \cE_{\lambda,k} - \dfrac{2 \left( \alpha - 2 \right)}{\alpha - 1} \gamma k^{2} \left\langle z_{k} - z_{k-1}, F \left( z_{k} \right) - F \left( w_{k-1} \right) \right\rangle\\
		\nonumber
		&\quad\phantom{:} + \dfrac{\alpha - 2}{\alpha - 1} \gamma^{2} k \sqrt{k} \left( \left( 1 - \varepsilon \right) \sqrt{k} + \alpha \right) \left\lVert v_{k} - v_{k-1} \right\rVert^{2}.
\end{align}

\subsection{Properties of the Energy Functions}
We collect the properties of the energy functions $\left(\cE_{\lambda,k}  \right)_{k \geq 1}$ and $\left(\cG_{\lambda,k}  \right)_{k \geq 1}$ in the following results. Please note that the statement of the first lemma can be deduced from eq. (81) in \cite{bot2022fastOGDA}, taking into account an appropriate formal correspondence between certain quantities. For better comprehensibility and in order to be able to refer to equations later on, we provide its proof nevertheless.

\begin{lemma}
	\label{lem:ene}
	Let $z^{\ast} \in \Omega$ and $\left(z_{k} \right)_{k \geq 0}$ be the sequence generated by Algorithm~\ref{algo:split}. For $0 \leq \lambda \leq \alpha - 1$, the following identity holds for every $k \geq 1$
	\begin{equation}\label{split:dE}
	\begin{split}
		\MoveEqLeft \cE_{\lambda,k+1} - \cE_{\lambda,k}\\
&= \minus 4 \left( \alpha - 2 \right) \lambda \gamma \left\langle z_{k+1} - z^{\ast}, v_{k+1} \right\rangle\\
		&\quad+ 2 \left( \lambda + 1 - \alpha \right) \left( 2k + \alpha + 1 \right) \left\lVert z_{k+1} - z_{k} \right\rVert ^{2} \\
		& \quad + \dfrac{2}{\alpha - 1} \Bigl( 4 \left( \alpha - 1 \right) \left( \lambda + 1 - \alpha \right) - \alpha \left( \alpha - 2 \right) \Bigr) \gamma k \left\langle z_{k+1} - z_{k}, v_{k+1} \right\rangle \\
		& \quad + \dfrac{2}{\alpha - 1} \Bigl( 2 \alpha \left( \alpha - 1 \right) \left( \lambda + 1 - \alpha \right) + \alpha - 2 \left( \alpha - 1 \right) ^{2} \Bigr) \gamma \left\langle z_{k+1} - z_{k}, v_{k+1} \right\rangle \\
		& \quad - \dfrac{2 \left( \alpha - 2 \right)}{\alpha - 1} \gamma k \left( k + \alpha \right) \left\langle z_{k+1} - z_{k}, v_{k+1} - v_{k} \right\rangle - \dfrac{\alpha - 2}{\alpha - 1} \gamma^{2} k \left( 2k + \alpha \right) \left\lVert v_{k+1} - v_{k} \right\rVert ^{2} \\
		& \quad - \dfrac{\alpha - 2}{\alpha - 1} \gamma^{2} \left( 2 \left( 3 \alpha - 2 \right) k + 2 \alpha^{2} + \alpha - 2 \right) \left\lVert v_{k+1} \right\rVert ^{2} .
	\end{split}
	\end{equation}
\end{lemma}

\begin{proof}
	Recall that by the definition in \eqref{split:defi:u-k-1-lambda}, we have for every $k \geq 1$
	\begin{equation}
		u_{\lambda,k}
		= 2 \lambda \left( z_{k} - z^{\ast} \right) + 2k \left( z_{k} - z_{k-1} \right) + \dfrac{3 \alpha - 2}{\alpha - 1} \gamma k v_{k}.
	\end{equation}
	Similarly,
	\begin{equation}\label{split:defi:u-k-lambda}
		u_{\lambda,k+1}
		= 2 \lambda \left( z_{k+1} - z^{\ast} \right) + 2 \left( k + 1 \right) \left( z_{k+1} - z_{k} \right) + \dfrac{3 \alpha - 2}{\alpha - 1} \gamma \left( k + 1 \right) v_{k+1} .
	\end{equation}
	Thus, after subtraction we deduce from \eqref{split:d-u} that
	\begin{equation}\label{split:defi:u-k-lambda:dif}
	\begin{split}
		\MoveEqLeft u_{\lambda,k+1} - u_{\lambda,k} \\
		&= 2 \left( \lambda + 1 - \alpha \right) \left( z_{k+1} - z_{k} \right) + 2 \left( k + \alpha \right) \left( z_{k+1} - z_{k} \right) - 2k \left( z_{k} - z_{k-1} \right) \\
			&\quad + \dfrac{3 \alpha - 2}{\alpha - 1} \gamma v_{k+1} + \dfrac{3 \alpha - 2}{\alpha - 1} \gamma k \left( v_{k+1} - v_{k} \right) \\
		&= 2 \left( \lambda + 1 - \alpha \right) \left( z_{k+1} - z_{k} \right) + \dfrac{\alpha - 2 \left( \alpha - 1 \right)^{2}}{\alpha - 1} \gamma v_{k+1} + \dfrac{2 - \alpha}{\alpha - 1} \gamma k \left( v_{k+1} - v_{k} \right) .
	\end{split}
	\end{equation}
	For $ k \geq 1 $ we know that
	\begin{equation}\label{split:dif:u-lambda:pre}
		\dfrac{1}{2} \left( \left\lVert u_{\lambda,k+1} \right\rVert ^{2} - \left\lVert u_{\lambda,k} \right\rVert ^{2} \right)
		= \left\langle u_{\lambda,k+1}, u_{\lambda,k+1} - u_{\lambda,k} \right\rangle - \dfrac{1}{2} \left\lVert u_{\lambda,k+1} - u_{\lambda,k} \right\rVert ^{2},
	\end{equation}
	and that for every $ k \geq 0 $
	\begin{equation}\label{split:dif:norm}
	\begin{split}
		\MoveEqLeft 2 \lambda \left( \alpha - 1 - \lambda \right) \left( \left\lVert z_{k+1} - z^{\ast} \right\rVert ^{2} - \left\lVert z_{k} - z^{\ast} \right\rVert ^{2} \right) \\
		&= 4 \lambda \left( \alpha - 1 - \lambda \right) \left\langle z_{k+1} - z^{\ast}, z_{k+1} - z_{k} \right\rangle - 2 \lambda \left( \alpha - 1 - \lambda \right) \left\lVert z_{k+1} - z_{k} \right\rVert ^{2}.
	\end{split}
	\end{equation}
	We use the relations \eqref{split:defi:u-k-lambda} and \eqref{split:defi:u-k-lambda:dif} to derive for every $k \geq 1$
	\begin{equation}\label{split:dif:u-lambda:inn}
	\begin{split}
		\MoveEqLeft \left\langle u_{\lambda,k+1}, u_{\lambda,k+1} - u_{\lambda,k} \right\rangle\\
		&= 4 \lambda \left( \lambda + 1 - \alpha \right) \left\langle z_{k+1} - z^{\ast}, z_{k+1} - z_{k} \right\rangle\\
		& \quad + \dfrac{2}{\alpha - 1} \left( \alpha - 2 \left( \alpha - 1 \right) ^{2} \right) \lambda \gamma \left\langle z_{k+1} - z^{\ast}, v_{k+1} \right\rangle \\
		& \quad - \dfrac{2 \left( \alpha - 2 \right)}{\alpha - 1} \lambda \gamma k \left\langle z_{k+1} - z^{\ast}, v_{k+1} - v_{k} \right\rangle
		+ 4 \left( \lambda + 1 - \alpha \right) \left( k + 1 \right) \left\lVert z_{k+1} - z_{k} \right\rVert ^{2} \\
		& \quad + \dfrac{2}{\alpha - 1} \Bigl( \left( 3 \alpha - 2 \right) \left( \lambda + 1 - \alpha \right) + \alpha - 2 \left( \alpha - 1 \right) ^{2} \Bigr) \gamma \left( k + 1 \right) \left\langle z_{k+1} - z_{k}, v_{k+1} \right\rangle \\
		& \quad - \dfrac{2 \left( \alpha - 2 \right)}{\alpha - 1} \gamma \left( k + 1 \right) k \left\langle z_{k+1} - z_{k}, v_{k+1} - v_{k} \right\rangle \\
		& \quad + \dfrac{1}{\left( \alpha - 1 \right) ^{2}} \left( \alpha - 2 \left( \alpha - 1 \right) ^{2} \right) \left( 3 \alpha - 2 \right) \gamma^{2} \left( k + 1 \right) \left\lVert v_{k+1} \right\rVert ^{2} \\
		& \quad - \dfrac{1}{\left( \alpha - 1 \right) ^{2}} \left( \alpha - 2 \right) \left( 3 \alpha - 2 \right) \gamma^{2} \left( k + 1 \right) k \left\langle v_{k+1}, v_{k+1} - v_{k} \right\rangle,
	\end{split}
	\end{equation}
	and
	\begin{equation}\label{split:dif:u-lambda:norm}
	\begin{split}
		\MoveEqLeft \minus \dfrac{1}{2} \left\lVert u_{\lambda,k+1} - u_{\lambda,k} \right\rVert ^{2}
		= \minus 2 \left( \lambda + 1 - \alpha \right) ^{2} \left\lVert z_{k+1} - z_{k} \right\rVert ^{2}\\
		&- \dfrac{2}{\alpha - 1} \left( \alpha - 2 \left( \alpha - 1 \right) ^{2} \right) \left( \lambda + 1 - \alpha \right) \gamma \left\langle z_{k+1} - z_{k}, v_{k+1} \right\rangle \\
		&- \dfrac{1}{2 \left( \alpha - 1 \right) ^{2}} \left( \alpha - 2 \left( \alpha - 1 \right) ^{2} \right) ^{2} \gamma^{2} \left\lVert v_{k+1} \right\rVert ^{2}
		- \dfrac{\left( \alpha - 2 \right)^{2}}{2 \left( \alpha - 1 \right)^{2}} \gamma^{2} k^{2} \left\lVert v_{k+1} - v_{k} \right\rVert ^{2} \\
		&+ \dfrac{2 (\alpha - 2)}{\alpha - 1} \left( \lambda + 1 - \alpha \right) \gamma k \left\langle z_{k+1} - z_{k}, v_{k+1} - v_{k} \right\rangle \\
		&+ \dfrac{\alpha - 2}{\left( \alpha - 1 \right) ^{2}} \left( \alpha - 2 \left( \alpha - 1 \right) ^{2} \right) \gamma^{2} k \left\langle v_{k+1}, v_{k+1} - v_{k} \right\rangle .
	\end{split}
	\end{equation}
	After some algebra, we see that
	\begin{equation}\label{split:al}
	\begin{split}
		\MoveEqLeft \Bigl( \left( 3 \alpha - 2 \right) \left( \lambda + 1 - \alpha \right) + \alpha - 2 \left( \alpha - 1 \right) ^{2} \Bigr) \left( k + 1 \right) - \left( \lambda + 1 - \alpha \right) \left( \alpha - 2 \left( \alpha - 1 \right) ^{2} \right) \\
		&= \Bigl( \left( 3 \alpha - 2 \right) \left( \lambda + 1 - \alpha \right) + \alpha - 2 \left( \alpha - 1 \right) ^{2} \Bigr) k + 2 \alpha \left( \alpha - 1 \right) \left( \lambda + 1 - \alpha \right)\\
			&\quad + \alpha - 2 \left( \alpha - 1 \right) ^{2} \\
		&= \Bigl( \left( 3 \alpha - 2 \right) \left( \lambda + 1 - \alpha \right) + \alpha - 2 \left( \alpha - 1 \right) ^{2} + \left( \alpha - 2 \right) \lambda \Bigr) k - \left( \alpha - 2 \right) \lambda k \\
			&\quad + 2 \alpha \left( \alpha - 1 \right) \left( \lambda + 1 - \alpha \right) + \alpha - 2 \left( \alpha - 1 \right) ^{2} \\
		&= \Bigl( 4 \left( \alpha - 1 \right) \left( \lambda + 1 - \alpha \right) - \alpha \left( \alpha - 2 \right) \Bigr) k - \left( \alpha - 2 \right) \lambda k\\
			&\quad + 2 \alpha \left( \alpha - 1 \right) \left( \lambda + 1 - \alpha \right) + \alpha - 2 \left( \alpha - 1 \right) ^{2}.
	\end{split}
	\end{equation}
	By plugging \eqref{split:dif:u-lambda:inn} and \eqref{split:dif:u-lambda:norm} into \eqref{split:dif:u-lambda:pre}, and by taking into consideration the relation \eqref{split:al}, we get for every $k \geq 1$
	\begin{equation}\label{split:dif:u-lambda}
	\begin{split}
		\MoveEqLeft \dfrac{1}{2} \left( \left\lVert u_{\lambda,k+1} \right\rVert ^{2} - \left\lVert u_{\lambda,k} \right\rVert ^{2} \right)
		=  4 \lambda \left( \lambda + 1 - \alpha \right) \left\langle z_{k+1} - z^{\ast}, z_{k+1} - z_{k} \right\rangle \\
		& + \dfrac{2}{\alpha - 1} \left( \alpha - 2 \left( \alpha - 1 \right) ^{2} \right) \lambda \gamma \left\langle z_{k+1} - z^{\ast}, v_{k+1} \right\rangle \\
		& - \dfrac{2 \left( \alpha - 2 \right)}{\alpha - 1} \lambda \gamma k \left\langle z_{k+1} - z^{\ast}, v_{k+1} - v_{k} \right\rangle\\
		& + 2 \left( \lambda + 1 - \alpha \right) \left( 2k + \alpha + 1 - \lambda \right) \left\lVert z_{k+1} - z_{k} \right\rVert ^{2} \\
		& + \dfrac{2}{\alpha - 1} \Bigl( 4 \left( \alpha - 1 \right) \left( \lambda + 1 - \alpha \right) - \alpha \left( \alpha - 2 \right) - \left( \alpha - 2 \right) \lambda \Bigr) \gamma k \left\langle z_{k+1} - z_{k}, v_{k+1} \right\rangle \\
		& + \dfrac{2}{\alpha - 1} \Bigl( 2 \alpha \left( \alpha - 1 \right) \left( \lambda + 1 - \alpha \right) + \alpha - 2 \left( \alpha - 1 \right) ^{2} \Bigr) \gamma \left\langle z_{k+1} - z_{k}, v_{k+1} \right\rangle \\
		& - \dfrac{2 \left( \alpha - 2 \right)}{\alpha - 1} \gamma k \left( k + \alpha - \lambda \right) \left\langle z_{k+1} - z_{k}, v_{k+1} - v_{k} \right\rangle
		- \dfrac{\left( \alpha - 2 \right)^{2}}{2 \left( \alpha - 1 \right)^{2}} \gamma^{2} k^{2} \left\lVert v_{k+1} - v_{k} \right\rVert ^{2} \\
		& + \dfrac{1}{2 \left( \alpha - 1 \right)^{2}} \left( \alpha - 2 \left( \alpha - 1 \right) ^{2} \right) \gamma^{2} \left( 2 \left( 3 \alpha - 2 \right) k + 2 \alpha^{2} + \alpha - 2 \right) \left\lVert v_{k+1} \right\rVert ^{2} \\
		& - \dfrac{\alpha - 2}{\left( \alpha - 1 \right) ^{2}} \gamma^{2} k \bigl( \left( 3 \alpha - 2 \right) k + 2 \alpha \left( \alpha - 1 \right) \bigr) \left\langle v_{k+1}, v_{k+1} - v_{k} \right\rangle.
	\end{split}
	\end{equation}
	Furthermore, one can show that for every $k \geq 1$ we get
	\begin{equation}\label{split:dif:vi}
	\begin{split}
		\MoveEqLeft \left( k + 1 \right) \left\langle z_{k+1} - z^{\ast}, v_{k+1} \right\rangle - k \left\langle z_{k} - z^{\ast}, v_{k} \right\rangle \\
		&= \left\langle z_{k+1} - z^{\ast}, v_{k+1} \right\rangle
		+ k \left( \left\langle z_{k+1} - z^{\ast}, v_{k+1} \right\rangle - \left\langle z_{k} - z^{\ast}, v_{k} \right\rangle \right) \\
		&= \left\langle z_{k+1} - z^{\ast}, v_{k+1} \right\rangle
		+ k \left\langle z_{k+1} - z^{\ast}, v_{k+1} - v_{k} \right\rangle \\
		& \qquad - k \left\langle z_{k+1} - z_{k}, v_{k+1} - v_{k} \right\rangle + k \left\langle z_{k+1} - z_{k}, v_{k+1} \right\rangle
	\end{split}
	\end{equation}
	and
	\begin{equation}\label{split:dif:eq}
	\begin{split}
		\MoveEqLeft \left( k + 1 \right) \Bigl( \left( 3 \alpha - 2 \right) \left( k+1 \right) + 2 \alpha \left( \alpha - 1 \right) \Bigr) \left\lVert v_{k+1} \right\rVert ^{2} - k \Bigl( \left( 3 \alpha - 2 \right) k + 2 \alpha \left( \alpha - 1 \right) \Bigr) \left\lVert v_{k} \right\rVert ^{2} \\
		&= \left( 2 \left( 3 \alpha - 2 \right) k + 2 \alpha^{2} + \alpha - 2 \right) \left\lVert v_{k+1} \right\rVert ^{2}\\
			&\quad+ k \Bigl( \left( 3 \alpha - 2 \right) k + 2 \alpha \left( \alpha - 1 \right) \Bigr) \left( \left\lVert v_{k+1} \right\rVert ^{2} - \left\lVert v_{k} \right\rVert ^{2} \right) \\
		&= \left( 2 \left( 3 \alpha - 2 \right) k + 2 \alpha^{2} + \alpha - 2 \right) \left\lVert v_{k+1} \right\rVert ^{2}\\
			&\quad + 2 k \Bigl( \left( 3 \alpha - 2 \right) k + 2 \alpha \left( \alpha - 1 \right) \Bigr) \left\langle v_{k+1}, v_{k+1} - v_{k} \right\rangle \\
			&\quad - k \Bigl( \left( 3 \alpha - 2 \right) k + 2 \alpha \left( \alpha - 1 \right) \Bigr) \left\lVert v_{k+1} - v_{k} \right\rVert^{2}.
	\end{split}
	\end{equation}
	In addition, direct computations show that
	\begin{equation}
		\alpha - 2 \left( \alpha - 1 \right) ^{2} + \alpha - 2
		= \minus 2 \left( \alpha - 1 \right) \left( \alpha - 2 \right)
	\end{equation}
	and
	\begin{equation}
		\minus \left( \alpha - 2 \right) ^{2} - \left( \alpha - 2 \right) \left( 3 \alpha - 2 \right)
		= \minus 4 \left( \alpha - 2 \right) \left( \alpha - 1 \right).
	\end{equation}
	Hence, multiplying \eqref{split:dif:vi} by $ \nicefrac{2 \lambda \gamma (\alpha - 2)}{(\alpha - 1)} \geq 0 $ and \eqref{split:dif:eq} by $\nicefrac{\gamma^{2} (\alpha - 2)}{2 (\alpha - 1)^{2}} > 0$, followed by summing up the resulting identities with \eqref{split:dif:norm} and \eqref{split:dif:u-lambda}, yields \eqref{split:dE} for every $k \geq 1$.
\end{proof}

\begin{lemma}
	\label{lem:reg}
	Let $z^{\ast} \in \Omega$ and $\left(z_{k} \right)_{k \geq 0}$ be the sequence generated by Algorithm~\ref{algo:split}. For $0 \leq \lambda \leq \alpha - 1$, the following statements are true:
	\begin{enumerate}[(i)]
		\item for every $k \geq k_{0} := \max \left\lbrace 2, \left\lceil \frac{1}{\alpha - 2} \right\rceil \right\rbrace$ the following holds:
		\begin{equation}
		\begin{split}
			\MoveEqLeft \cG_{\lambda,k+1} - \cG_{\lambda,k}
			\leq \dfrac{\left( \alpha - 1 \right) \left( \alpha - 2 \right)}{\varepsilon \left( k+1 \right) ^{2}}  \lambda ^{2} \left\lVert z_{k+1} - z^{\ast} \right\rVert ^{2}\\
				&- 4 \left( \alpha - 2 \right) \lambda \gamma \left\langle z_{k+1} - z^{\ast}, \zeta_{k+1} + F \left( z_{k+1} \right) \right\rangle \\
				&+ 4 \left( \eta_{0} k + \eta_{1} \right) \gamma \left\langle z_{k+1} - z_{k}, v_{k+1} \right\rangle
				+ \left( \eta_{2} k + \kappa_{0} \sqrt{k} \right) \left\lVert z_{k+1} - z_{k} \right\rVert ^{2} \\
				&+ 4 \left( \eta_{3} k + \kappa_{1} \sqrt{k} \right) \gamma^{2} \left\lVert v_{k+1} \right\rVert ^{2} - \dfrac{\alpha - 2}{\alpha - 1} \mu_{k} \gamma^{2} \left\lVert v_{k+1} - v_{k} \right\rVert ^{2},
		\end{split}
		\end{equation}
		where
		\begin{equation}\label{ene:const}
		\begin{split}
			\mu_{k} & := (k+1) \left( \varepsilon (k+1) + \alpha^{2} \sqrt{k+1} + (\alpha - 4) \right) - (\alpha - 2), \\[1.5ex]
			\eta_{0} 	& := \dfrac{1}{2 \left( \alpha - 1 \right)} \left( 4 \left( \alpha - 1 \right) \left( \lambda + 1 - \alpha \right) - \alpha \left( \alpha - 2 \right) \right) < 0, \\
			\eta_{1} 	& := \dfrac{1}{2 \left( \alpha - 1 \right)} \left( 2 \alpha \left( \alpha - 1 \right) \left( \lambda + 1 - \alpha \right) + \alpha - 2 \left( \alpha - 1 \right) ^{2} \right) < 0, \\
			\eta_{2} 	& := 4 \left( \lambda + 1 - \alpha \right) \leq 0, \\
			\eta_{3} 	& := - \dfrac{1}{2 \left( \alpha - 1 \right)} \left( \alpha - 2 \right) \left( 3 \alpha - 2 \right) < 0, \\[1.5ex]
			\kappa_{0} 	& := \dfrac{1}{\alpha - 1} \left( \alpha - 2 \right) \sqrt{\alpha - 2} > 0, \\
			\kappa_{1} 	& := \dfrac{1}{4 \left( \alpha - 1 \right)} \left( \alpha - 2 \right) \alpha > 0;
		\end{split}
		\end{equation}

		\item for every $k \geq 1$ one has the following lower bound for the quantity $ \cG_{\lambda, k} $
		\begin{equation}\label{ene:low}
		\begin{split}
			\MoveEqLeft \cG_{\lambda,k}
			\geq \dfrac{\alpha - 2}{4 \left( 3 \alpha - 2 \right)} \left\lVert 4 \lambda \left( z_{k} - z^{\ast} \right) + 2 k \left( z_{k} - z_{k-1} \right) + \dfrac{2 (3 \alpha - 2)}{\alpha - 1} \gamma k v_{k} \right\rVert ^{2} \\
			& + \dfrac{\left( \alpha - 2 \right) ^{2}}{4 \left( 3 \alpha - 2 \right) \left( \alpha - 1 \right) } k^{2} \left\lVert z_{k} - z_{k-1} \right\rVert ^{2} + 2 \left( \alpha - 1 \right) \lambda \left( 1 - \dfrac{4 \lambda}{3 \alpha - 2} \right) \left\lVert z_{k} - z^{\ast} \right\rVert ^{2}.
		\end{split}
		\end{equation}
	\end{enumerate}
\end{lemma}
\begin{proof}
	(i) Let $k \geq 2$ be fixed.
	By the definition of $\cG_{\lambda,k} $ in \eqref{split:defi:F}, we have for every $k \geq 2$
	\begin{equation}\label{ene:dF}
	\begin{split}
		\MoveEqLeft \cG_{\lambda,k+1}  - \cG_{\lambda,k} \\
		&= \cE_{\lambda,k+1} - \cE_{\lambda,k} - \dfrac{2 \left( \alpha - 2 \right)}{\alpha - 1} \gamma \left[ \left( k + 1 \right) ^{2} \left\langle z_{k+1} - z_{k}, F \left( z_{k+1} \right) - F \left( w_{k} \right) \right\rangle \right.\\
		&\left. \phantom{\left( k+1 \right)^{2}} - k^{2} \left\langle z_{k} - z_{k-1}, F \left( z_{k} \right) - F \left( w_{k-1} \right) \right\rangle \right] \\
		&\quad + \dfrac{\left( \alpha - 2 \right) \alpha}{\alpha - 1}  \gamma^{2} \left[ \left( k+1 \right) \sqrt{k+1} \left\lVert v_{k+1} - v_{k} \right\rVert ^{2}
		- k \sqrt{k} \left\lVert v_{k} - v_{k-1} \right\rVert ^{2} \right] \\
		&\quad + \dfrac{\left( \alpha - 2 \right) \left( 1 - \varepsilon \right)}{\alpha - 1} \gamma^{2} \left[\left( k+1 \right)^{2} \left\lVert v_{k+1} - v_{k} \right\rVert ^{2}
		- k^{2} \left\lVert v_{k} - v_{k-1} \right\rVert ^{2} \right] .
	\end{split}
	\end{equation}
	By using the definition of $\eta_{0}, \eta_{1}, \eta_{2}$ and $\eta_{3}$ in \eqref{ene:const}, for every $k \geq 1$ we deduce from \eqref{split:dE} that
	\begin{equation}\label{ene:dE}
	\begin{split}
		\MoveEqLeft \cE_{\lambda,k+1} - \cE_{\lambda,k} \\
		&= \minus 4 \left( \alpha - 2 \right) \lambda \gamma \left\langle z_{k+1} - z^{\ast}, v_{k+1} \right\rangle
		+ \Bigl( \eta_{2} k + 2 \left( \lambda + 1 - \alpha \right) \left( \alpha + 1 \right) \Bigr) \left\lVert z_{k+1} - z_{k} \right\rVert ^{2} \\
		&\quad + 4 \left( \eta_{0} k + \eta_{1} \right) \gamma \left\langle z_{k+1} - z_{k}, v_{k+1} \right\rangle
		- \dfrac{2 \left( \alpha - 2 \right)}{\alpha - 1} \gamma k \left( k + \alpha \right) \left\langle z_{k+1} - z_{k}, v_{k+1} - v_{k} \right\rangle \\
		&\quad - \dfrac{\alpha - 2}{\alpha - 1} k \left( 2k + \alpha \right) \gamma^{2} \left\lVert v_{k+1} - v_{k} \right\rVert ^{2}
		+ \left( 4 \eta_{3} k - \dfrac{\alpha - 2}{\alpha - 1} \left( 2 \alpha^{2} + \alpha - 2 \right) \right) \gamma^{2} \left\lVert v_{k+1} \right\rVert ^{2}\\
		&\leq \minus 4 \left( \alpha - 2 \right) \lambda \gamma \left\langle z_{k+1} - z^{\ast}, v_{k+1} \right\rangle
		- \dfrac{2 \left( \alpha - 2 \right)}{\alpha - 1} \gamma k \left( k + \alpha \right) \left\langle z_{k+1} - z_{k}, v_{k+1} - v_{k} \right\rangle \\
		&\quad + \left( 4 \left( \eta_{0} k + \eta_{1} \right) \gamma \left\langle z_{k+1} - z_{k}, v_{k+1} \right\rangle
		+ \eta_{2} k \left\lVert z_{k+1} - z_{k} \right\rVert ^{2}
		+ 4 \eta_{3} k \gamma^{2} \left\lVert v_{k+1} \right\rVert ^{2} \right) \\
		&\quad - \dfrac{\alpha - 2}{\alpha - 1} k \left( 2k + \alpha \right) \gamma^{2} \left\lVert v_{k+1} - v_{k} \right\rVert ^{2},
	\end{split}
	\end{equation}
	where the inequality comes from the fact that $0 \leq \lambda \leq \alpha - 1$ and $\alpha > 2$.
	Plugging \eqref{ene:dE} into \eqref{ene:dF} yields for every $k \geq 2$
	\begin{equation}\label{ene:pre}
	\begin{split}
		\MoveEqLeft \cG_{\lambda,k+1}  - \cG_{\lambda,k} \\
		&\leq \minus 4 \left( \alpha - 2 \right) \lambda \gamma \left\langle z_{k+1} - z^{\ast}, v_{k+1} \right\rangle
		- \dfrac{\alpha - 2}{\alpha - 1} \left( 2k^{2} + \alpha k \right) \gamma^{2} \left\lVert v_{k+1} - v_{k} \right\rVert ^{2} \\
			&\quad - \dfrac{2 \left( \alpha - 2 \right)}{\alpha - 1} \gamma \left[ \left( k + 1 \right) ^{2} \left\langle z_{k+1} - z_{k}, F \left( z_{k+1} \right) - F \left( w_{k} \right) \right\rangle \right.\\
			&\left. \phantom{\left( k + 1 \right)^{2}}- k^{2} \left\langle z_{k} - z_{k-1}, F \left( z_{k} \right) - F \left( w_{k-1} \right) \right\rangle \right] \\
			&\quad + \dfrac{\left( \alpha - 2 \right) \alpha}{\alpha - 1} \gamma^{2} \left[ \left( k+1 \right) \sqrt{k+1} \left\lVert v_{k+1} - v_{k} \right\rVert ^{2}
			- k \sqrt{k} \left\lVert v_{k} - v_{k-1} \right\rVert ^{2} \right] \\
			&\quad + \dfrac{\left( \alpha - 2 \right) \left( 1 - \varepsilon \right)}{\alpha - 1} \gamma^{2} \left[\left( k+1 \right)^{2} \left\lVert v_{k+1} - v_{k} \right\rVert ^{2}
			- k^{2} \left\lVert v_{k} - v_{k-1} \right\rVert ^{2} \right] \\
			&\quad + \left( 4 \left( \eta_{0} k + \eta_{1} \right) \gamma \left\langle z_{k+1} - z_{k}, v_{k+1} \right\rangle
			+ \eta_{2} k \left\lVert z_{k+1} - z_{k} \right\rVert ^{2}
			+ 4 \eta_{3} k \gamma^{2} \left\lVert v_{k+1} \right\rVert ^{2} \right) \\
			&\quad - \dfrac{2 \left( \alpha - 2 \right)}{\alpha - 1} \gamma k \left( k + \alpha \right) \left\langle z_{k+1} - z_{k}, v_{k+1} - v_{k} \right\rangle.
	\end{split}
	\end{equation}
	Our next aim is to derive upper estimates for the first two terms on the right-hand side of \eqref{ene:pre}, which will eventually simplify the subsequent three terms.
	From the Cauchy-Schwarz inequality and \eqref{split:Lip} we have for every $k \geq 1$
	\begin{equation}\label{ene:inn}
	\begin{split}
		\MoveEqLeft \minus 4 \left( \alpha - 2 \right) \lambda \gamma \left\langle z_{k+1} - z^{\ast}, v_{k+1} \right\rangle
		= \minus 4 \left( \alpha - 2 \right) \lambda \gamma \left\langle z_{k+1} - z^{\ast}, \zeta_{k+1} + F \left( w_{k} \right) \right\rangle \\
		&= \minus 4 \left( \alpha - 2 \right) \lambda \gamma \left\langle z_{k+1} - z^{\ast}, \zeta_{k+1} + F \left( z_{k+1} \right) \right\rangle\\
			&\quad + 4 \left( \alpha - 2 \right) \lambda \gamma \left\langle z_{k+1} - z^{\ast}, F \left( z_{k+1} \right) - F \left( w_{k} \right) \right\rangle \\
		&\leq \minus 4 \left( \alpha - 2 \right) \lambda \gamma \left\langle z_{k+1} - z^{\ast}, \zeta_{k+1} + F \left( z_{k+1} \right) \right\rangle\\
			&\quad + 4 \left( \alpha - 2 \right) \lambda \gamma \left\lVert z_{k+1} - z^{\ast} \right\rVert \left\lVert F \left( z_{k+1} \right) - F \left( w_{k} \right) \right\rVert \\
		&\leq \minus 4 \left( \alpha - 2 \right) \lambda \gamma \left\langle z_{k+1} - z^{\ast}, \zeta_{k+1} + F \left( z_{k+1} \right) \right\rangle\\
			&\quad + 2 \left( \alpha - 2 \right) \lambda \gamma \left\lVert z_{k+1} - z^{\ast} \right\rVert \left\lVert v_{k+1} - v_{k} \right\rVert \\
		&\leq \minus 4 \left( \alpha - 2 \right) \lambda \gamma \left\langle z_{k+1} - z^{\ast}, \zeta_{k+1} + F \left( z_{k+1} \right) \right\rangle
		+ \dfrac{\left( \alpha - 1 \right) \left( \alpha - 2 \right)}{\varepsilon \left( k+1 \right) ^{2}} \lambda ^{2} \left\lVert z_{k+1} - z^{\ast} \right\rVert ^{2} \\
			&\quad + \dfrac{\alpha - 2}{\alpha - 1} \varepsilon \gamma^{2} \left( k+1 \right) ^{2} \left\lVert v_{k+1} - v_{k} \right\rVert ^{2} .
	\end{split}
	\end{equation}
	For $\zeta_{k} \in N_{C} \left( z_{k} \right)$ and $\zeta_{k+1} \in N_{C} \left( z_{k+1} \right)$, the monotonicity of $N_{C}$ and $F$, together with the relation~\eqref{split:d-u} and the fact that for every $ k \geq 1 $
	\begin{equation}
		\zeta_{k} + F(z_{k}) - v_{k} = F(z_{k}) - F(w_{k-1}),
	\end{equation}
	yield for every $k \geq 1$
	\begin{equation}\label{ene:mono}
	\begin{split}
		\MoveEqLeft \minus \dfrac{2 \left( \alpha - 2 \right)}{\alpha - 1} \gamma k \left( k + \alpha \right) \left\langle z_{k+1} - z_{k}, v_{k+1} - v_{k} \right\rangle \\
		&\leq \dfrac{2 \left( \alpha - 2 \right)}{\alpha - 1} \gamma k \left( k + \alpha \right) \left\langle z_{k+1} - z_{k}, \Bigl( \zeta_{k+1} + F \left( z_{k+1} \right) - v_{k+1} \Bigr) - \Bigl( \zeta_{k} + F \left( z_{k} \right) - v_{k} \Bigr) \right\rangle \\
		&= \dfrac{2 \left( \alpha - 2 \right)}{\alpha - 1} \gamma k \left( k + \alpha \right) \left\langle z_{k+1} - z_{k}, \Bigl( F \left( z_{k+1} \right) - F \left( w_{k} \right) \Bigr) - \Bigl( F \left( z_{k} \right) - F \left( w_{k-1} \right) \Bigr) \right\rangle \\
		&= \dfrac{2 \left( \alpha - 2 \right)}{\alpha - 1} \gamma k \left( k + \alpha \right) \left\langle z_{k+1} - z_{k}, F \left( z_{k+1} \right) - F \left( w_{k} \right) \right\rangle \\
			&\quad - \dfrac{2 \left( \alpha - 2 \right)}{\alpha - 1} \gamma k \left( k + \alpha \right) \left\langle z_{k+1} - z_{k}, F \left( z_{k} \right) - F \left( w_{k-1} \right) \right\rangle \\
		&= \dfrac{2 \left( \alpha - 2 \right)}{\alpha - 1} \gamma \left( k + 1 \right) ^{2} \left\langle z_{k+1} - z_{k}, F \left( z_{k+1} \right) - F \left( w_{k} \right) \right\rangle\\
			&\quad - \dfrac{2 \left( \alpha - 2 \right)}{\alpha - 1} \gamma k^{2} \left\langle z_{k} - z_{k-1}, F \left( z_{k} \right) - F \left( w_{k-1} \right) \right\rangle \\
			&\quad + \dfrac{2 \left( \alpha - 2 \right)}{\alpha - 1} \gamma \bigl( \left( \alpha - 2 \right) k - 1 \bigr) \left\langle z_{k+1} - z_{k}, F \left( z_{k+1} \right) - F \left( w_{k} \right) \right\rangle \\
			&\quad + \dfrac{2 \left( \alpha - 2 \right)}{\alpha - 1} \alpha \gamma^{2} k \left\langle v_{k+1}, F \left( z_{k} \right) - F \left( w_{k-1} \right) \right\rangle \\
			&\quad + \dfrac{4 \left( \alpha - 2 \right)}{\alpha - 1} \gamma^{2} k^{2} \left\langle v_{k+1} - v_{k}, F \left( z_{k} \right) - F \left( w_{k-1} \right) \right\rangle .
	\end{split}
	\end{equation}
	By Young's inequality together with \eqref{split:Lip} for every $k \geq \left\lceil \frac{1}{\alpha - 2} \right\rceil$ we obtain
	\begin{equation}\label{ene:Young:1}
	\begin{split}
		\MoveEqLeft \dfrac{2 \left( \alpha - 2 \right)}{\alpha - 1} \gamma \bigl( \left( \alpha - 2 \right) k - 1 \bigr) \left\langle z_{k+1} - z_{k}, F \left( z_{k+1} \right) - F \left( w_{k} \right) \right\rangle \\
		&\leq \dfrac{\alpha - 2}{\alpha - 1} \left( \sqrt{\left( \alpha - 2 \right) k - 1} \left\lVert z_{k+1} - z_{k} \right\rVert ^{2} \right.\\
			&\hspace{2.5cm}\left.+ \gamma^{2} \bigl( \left( \alpha - 2 \right) k - 1 \bigr) \sqrt{\left( \alpha - 2 \right) k - 1} \left\lVert F \left( z_{k+1} \right) - F \left( w_{k} \right) \right\rVert ^{2} \right) \\
		&\leq \dfrac{\alpha - 2}{\alpha - 1} \sqrt{\left( \alpha - 2 \right) k} \left\lVert z_{k+1} - z_{k} \right\rVert ^{2} \\
			&\quad+ \left( \alpha - 2 \right) \sqrt{\alpha - 1} \gamma^{2} \left( k+1 \right) \sqrt{k+1} \left\lVert F \left( z_{k+1} \right) - F \left( w_{k} \right) \right\rVert ^{2} \\
		&\leq \dfrac{\alpha - 2}{\alpha - 1} \sqrt{\left( \alpha - 2 \right) k} \left\lVert z_{k+1} - z_{k} \right\rVert ^{2} \\
			&\quad+ 4 \left( \alpha - 2 \right) \sqrt{\alpha - 1} \gamma^{4} L^{2} \left( k+1 \right) \sqrt{k+1} \left\lVert v_{k+1} - v_{k} \right\rVert ^{2} \\
		&\leq \dfrac{\alpha - 2}{\alpha - 1} \sqrt{\left( \alpha - 2 \right) k} \left\lVert z_{k+1} - z_{k} \right\rVert ^{2} + \left( \alpha - 2 \right) \alpha \gamma^{2} \left( k+1 \right) \sqrt{k+1} \left\lVert v_{k+1} - v_{k} \right\rVert ^{2},
	\end{split}
	\end{equation}
	where in the second estimate we use the fact that $\left( \alpha - 2 \right) k -1 \leq \left( \alpha - 1 \right) \left( k + 1 \right)$, while in the last one we combine $\sqrt{\alpha - 1} \leq \alpha$ and $ \gamma L < \nicefrac{1}{4} < 1$.

	In addition, for every $k \geq 2$ we derive
	\begin{equation}\label{ene:Young:2}
	\begin{split}
		\MoveEqLeft \dfrac{2 \left( \alpha - 2 \right)}{\alpha - 1}  \alpha \gamma^{2} k \left\langle v_{k+1}, F \left( z_{k} \right) - F \left( w_{k-1} \right) \right\rangle \\
		&\leq \dfrac{\alpha - 2}{\alpha - 1} \alpha \gamma^{2} \sqrt{k} \left\lVert v_{k+1} \right\rVert ^{2} + \dfrac{\alpha - 2 }{\alpha - 1} \alpha \gamma^{2} k \sqrt{k} \left\lVert F \left( z_{k} \right) - F \left( w_{k-1} \right) \right\rVert ^{2} \\
		&\leq \dfrac{\alpha - 2}{\alpha - 1} \alpha \gamma^{2} \sqrt{k} \left\lVert v_{k+1} \right\rVert ^{2} + \dfrac{\alpha - 2}{\alpha - 1} \alpha \gamma^{2} k \sqrt{k} \left\lVert v_{k} - v_{k-1} \right\rVert ^{2} \\
		&= \dfrac{\alpha - 2}{\alpha - 1} \alpha \gamma^{2} \sqrt{k} \left\lVert v_{k+1} \right\rVert ^{2} + \dfrac{\alpha - 2}{\alpha - 1} \alpha \gamma^{2} \left( k+1 \right) \sqrt{k+1} \left\lVert v_{k+1} - v_{k} \right\rVert ^{2} \\
			&\quad - \dfrac{\alpha - 2}{\alpha - 1} \alpha \gamma^{2} \left[ \left( k+1 \right) \sqrt{k+1} \left\lVert v_{k+1} - v_{k} \right\rVert ^{2} - k \sqrt{k} \left\lVert v_{k} - v_{k-1} \right\rVert ^{2} \right] ,
	\end{split}
	\end{equation}
	and, by using the Cauchy-Schwarz inequality and \eqref{split:Lip},
	\begin{equation}\label{ene:C-S}
	\begin{split}
		\MoveEqLeft \dfrac{4 \left( \alpha - 2 \right)}{\alpha - 1} \gamma^{2} k^{2} \left\langle v_{k+1} - v_{k}, F \left( z_{k} \right) - F \left( w_{k-1} \right) \right\rangle \\
		&\leq \dfrac{2 \left( \alpha - 2 \right)}{\alpha - 1} \left( 1 - \varepsilon \right) \gamma^{2} k^{2} \left\lVert v_{k+1} - v_{k} \right\rVert \left\lVert v_{k} - v_{k-1} \right\rVert \\
		&\leq \dfrac{\alpha - 2}{\alpha - 1} \left( 1 - \varepsilon \right) \gamma^{2} k^{2} \left( \left\lVert v_{k+1} - v_{k} \right\rVert ^{2} + \left\lVert v_{k} - v_{k-1} \right\rVert ^{2} \right) \\
		&= \dfrac{4 \left( \alpha - 2 \right)}{\alpha - 1} \gamma^{3} Lk^{2} \left( \left\lVert v_{k+1} - v_{k} \right\rVert ^{2} + \left\lVert v_{k} - v_{k-1} \right\rVert ^{2} \right) \\
		&\leq \minus \dfrac{4 \left( \alpha - 2 \right)}{\alpha - 1} \gamma^{3} L \left( \left( k+1 \right) ^{2} \left\lVert v_{k+1} - v_{k} \right\rVert ^{2} - k^{2} \left\lVert v_{k} - v_{k-1} \right\rVert ^{2} \right) \\
			&\quad + \dfrac{8 \left( \alpha - 2 \right)}{\alpha - 1} \gamma^{3} L \left( k+1 \right) ^{2} \left\lVert v_{k+1} - v_{k} \right\rVert ^{2} \\
		&= \minus \dfrac{4 \left( \alpha - 2 \right)}{\alpha - 1} \gamma^{3} L \left( \left( k+1 \right) ^{2} \left\lVert v_{k+1} - v_{k} \right\rVert ^{2} - k^{2} \left\lVert v_{k} - v_{k-1} \right\rVert ^{2} \right) \\
			&\quad + \dfrac{2 \left( \alpha - 2 \right)}{\alpha - 1} \left( 1 - \varepsilon \right) \gamma^{2} \left( k+1 \right) ^{2} \left\lVert v_{k+1} - v_{k} \right\rVert ^{2} ,
	\end{split}
	\end{equation}
	where we want to recall that the first equality comes from \eqref{defi:s-e}.
	By plugging \eqref{ene:Young:1} and \eqref{ene:C-S} into \eqref{ene:mono}, then combining the result with \eqref{ene:inn}, we get after rearranging the terms for every $k \geq k_{0}$
	\begin{equation}\label{ene:sum}
	\begin{split}
		\MoveEqLeft \minus 4 \left( \alpha - 2 \right) \lambda \gamma \left\langle z_{k+1} - z^{\ast}, v_{k+1} \right\rangle
		- \dfrac{2 \left( \alpha - 2 \right)}{\alpha - 1} \gamma k \left( k + \alpha \right) \left\langle z_{k+1} - z_{k}, v_{k+1} - v_{k} \right\rangle \\
		&\leq \dfrac{2 \left( \alpha - 2 \right)}{\alpha - 1} \gamma \left[ \left( k + 1 \right) ^{2} \left\langle z_{k+1} - z_{k}, F \left( z_{k+1} \right) - F \left( w_{k} \right) \right\rangle \right. \\
		&\left. \phantom{ \left( k + 1 \right)^{2} } - k^{2} \left\langle z_{k} - z_{k-1}, F \left( z_{k} \right) - F \left( w_{k-1} \right) \right\rangle \right] \\
			&\quad - \dfrac{\alpha - 2}{\alpha - 1} \alpha \gamma^{2} \left[ \left( k+1 \right) \sqrt{k+1} \left\lVert v_{k+1} - v_{k} \right\rVert ^{2}
			- k \sqrt{k} \left\lVert v_{k} - v_{k-1} \right\rVert ^{2} \right] \\
			&\quad - \dfrac{4 \left( \alpha - 2 \right)}{\alpha - 1} \gamma^{3} L \left[\left( k+1 \right)^{2} \left\lVert v_{k+1} - v_{k} \right\rVert ^{2}
			- k^{2} \left\lVert v_{k} - v_{k-1} \right\rVert ^{2} \right] \\
			&\quad - 4 \left( \alpha - 2 \right) \lambda \gamma \left\langle z_{k+1} - z^{\ast}, \zeta_{k+1} + F \left( z_{k+1} \right) \right\rangle\\
			&\quad + \dfrac{\alpha - 2}{\alpha - 1} \left( \left( 2k^{2} + \alpha k \right) - \mu_{k} \right) \gamma^{2} \left\lVert v_{k+1} - v_{k} \right\rVert ^{2} \\
			&\quad + \dfrac{1}{\alpha - 1} \left( \alpha - 2 \right) \sqrt{\left( \alpha - 2 \right) k} \left\lVert z_{k+1} - z_{k} \right\rVert ^{2} + \dfrac{\alpha - 2}{\alpha - 1} \alpha \gamma^{2} \sqrt{k} \left\lVert v_{k+1} \right\rVert ^{2} \\
			&\quad + \dfrac{1}{\varepsilon \left( k+1 \right) ^{2}} \left( \alpha - 1 \right) \left( \alpha - 2 \right) \lambda ^{2} \left\lVert z_{k+1} - z^{\ast} \right\rVert ^{2},
	\end{split}
	\end{equation}
	where we set
	\begin{align*}
		\mu_{k}
			&:= \left( 2k^{2} + \alpha k \right) - \varepsilon \left( k+1 \right) ^{2} - \left( \alpha - 1 \right) \alpha \left( k+1 \right) \sqrt{k+1} - \alpha \left( k+1 \right) \sqrt{k+1} \\
			&\quad - 2 \left( 1 - \varepsilon \right) \left( k+1 \right) ^{2} \\
			&= \varepsilon \left( k+1 \right) ^{2} + \left( \alpha - 4 \right) k - 2 - \alpha^{2} \left( k+1 \right) \sqrt{k+1} \\
			&= \left( k+1 \right) \left( \varepsilon \left( k+1 \right) + \alpha^{2} \sqrt{k+1} + \alpha - 4 \right) - \left( \alpha - 2 \right) .
	\end{align*}
	Finally, summing up \eqref{ene:pre} and \eqref{ene:sum}, we obtain the desired estimate.

	(ii)
	Observe that
	\begin{equation}
	\begin{split}
		\MoveEqLeft \dfrac{2 \left( \alpha - 2 \right)}{\alpha - 1} \lambda \gamma k \left\langle z_{k} - z^{\ast}, v_{k} \right\rangle
		+ \dfrac{\left( \alpha - 2 \right) \left( 3 \alpha - 2 \right)}{2 \left( \alpha - 1 \right)^{2}}  \gamma^{2} k^{2} \left\lVert v_{k} \right\rVert ^{2} \nonumber \\
			&= \dfrac{\alpha - 2}{3 \alpha - 2} \left( \dfrac{2 \left( 3 \alpha - 2 \right)}{\alpha - 1} \lambda \gamma k \left\langle z_{k} - z^{\ast}, v_{k} \right\rangle + \dfrac{ \left( 3 \alpha - 2 \right)^{2}}{2 \left( \alpha - 1 \right)^{2}} \gamma^{2} k^{2} \left\lVert v_{k} \right\rVert ^{2} \right) \nonumber \\
			&= \dfrac{1}{3 \alpha - 2} \left( \alpha - 2 \right) \left( \dfrac{1}{2} \left\lVert 2 \lambda \left( z_{k} - z^{\ast} \right) + \dfrac{3 \alpha - 2}{\alpha - 1} \gamma k v_{k} \right\rVert ^{2} - 2 \lambda^{2} \left\lVert z_{k} - z^{\ast} \right\rVert ^{2} \right).
	\end{split}
	\end{equation}
	By the definition of $u_{\lambda,k}$ in \eqref{split:defi:u-k-1-lambda} and by using the identity
	\begin{equation}
		\left\lVert x \right\rVert ^{2} + \left\lVert y \right\rVert ^{2}
		= \frac{1}{2} \left( \left\lVert x + y \right\rVert ^{2} + \left\lVert x - y \right\rVert ^{2} \right)
		\quad \forall x, y \in \R^{d},
	\end{equation}
	we deduce that for every $k \geq 1$
	\begin{equation}
	\begin{split}
		\MoveEqLeft \cE_{\lambda,k}
		=\dfrac{1}{2} \left\lVert u_{\lambda,k} \right\rVert ^{2}
		+ 2 \lambda \left( \alpha - 1 - \lambda \right) \left\lVert z_{k} - z^{\ast} \right\rVert ^{2}
		+ \dfrac{2 \left( \alpha - 2 \right)}{\alpha - 1} \lambda \gamma k \left\langle z_{k} - z^{\ast}, v_{k} \right\rangle \\
			& \quad + \dfrac{\alpha - 2}{\alpha - 1} \gamma^{2} k \left( \dfrac{1}{2 \left( \alpha - 1 \right)} \left( 3 \alpha - 2 \right) k + \alpha \right) \left\lVert v_{k} \right\rVert ^{2} \nonumber \\
		&= \dfrac{1}{2} \left\lVert 2 \lambda \left( z_{k} - z^{\ast} \right) + 2k \left( z_{k} - z_{k-1} \right) + \dfrac{3 \alpha - 2}{\alpha - 1} \gamma k v_{k} \right\rVert ^{2}\\
			&\quad+ 2 \left( \alpha - 1 \right) \lambda \left( 1 - \dfrac{4 \lambda}{3 \alpha - 2} \right) \left\lVert z_{k} - z^{\ast} \right\rVert ^{2}
			+ \dfrac{\alpha - 2}{\alpha - 1} \alpha \gamma^{2} k \left\lVert v_{k} \right\rVert ^{2} \\
			&\quad+ \dfrac{\alpha - 2}{2 \left( 3 \alpha - 2 \right)} \left\lVert 2 \lambda \left( z_{k} - z^{\ast} \right) + \dfrac{3 \alpha - 2}{\alpha - 1} \gamma k v_{k} \right\rVert ^{2} \\
		&= \dfrac{\alpha}{3 \alpha - 2} \left\lVert 2 \lambda \left( z_{k} - z^{\ast} \right) + 2k \left( z_{k} - z_{k-1} \right) + \dfrac{3 \alpha - 2}{\alpha - 1} \gamma k v_{k} \right\rVert ^{2} \\
			&\quad + 2 \left( \alpha - 1 \right) \lambda \left( 1 - \dfrac{4 \lambda}{3 \alpha - 2} \right) \left\lVert z_{k} - z^{\ast} \right\rVert ^{2}
			+ \dfrac{\alpha - 2}{\alpha - 1} \alpha \gamma^{2} k \left\lVert v_{k} \right\rVert ^{2} \\
			&\quad + \dfrac{\alpha - 2}{4 \left( 3 \alpha - 2 \right)} \left\lVert 4 \lambda \left( z_{k} - z^{\ast} \right) + 2k \left( z_{k} - z_{k-1} \right) + \dfrac{2 \left( 3 \alpha - 2 \right)}{\alpha - 1} \gamma k v_{k} \right\rVert ^{2} \\
			&\quad + \dfrac{\alpha - 2}{3 \alpha - 2} k^{2} \left\lVert z_{k} - z_{k-1} \right\rVert ^{2} .
	\end{split}
	\end{equation}
	Consequently,
	\begin{align*}
	\cG_{\lambda,k}
	& = \cE_{\lambda,k} - \dfrac{2 \left( \alpha - 2 \right)}{\alpha - 1} \gamma k^{2} \left\langle z_{k} - z_{k-1}, F \left( z_{k} \right) - F \left( w_{k-1} \right) \right\rangle \nonumber \\
	& \qquad + \dfrac{\alpha - 2}{\alpha - 1} \gamma^{2} k \sqrt{k} \left( \left( 1 - \varepsilon \right) \sqrt{k} + \alpha \right) \left\lVert v_{k} - v_{k-1} \right\rVert ^{2} \nonumber \\
	& \geq \dfrac{\alpha - 2}{4 \left( 3 \alpha - 2 \right)} \left\lVert 4 \lambda \left( z_{k} - z^{\ast} \right) + 2k \left( z_{k} - z_{k-1} \right) + \dfrac{2 \left( 3 \alpha - 2 \right)}{\alpha - 1} \gamma k v_{k} \right\rVert ^{2} \nonumber \\
	& \qquad + \dfrac{\alpha - 2}{3 \alpha - 2} k^{2} \left\lVert z_{k} - z_{k-1} \right\rVert ^{2} + 2 \left( \alpha - 1 \right) \lambda \left( 1 - \dfrac{4 \lambda}{3 \alpha - 2} \right) \left\lVert z_{k} - z^{\ast} \right\rVert ^{2} \nonumber \\
	& \qquad - \dfrac{2 \left( \alpha - 2 \right)}{\alpha - 1} \gamma k^{2} \left\langle z_{k} - z_{k-1}, F \left( z_{k} \right) - F \left( w_{k-1} \right) \right\rangle \nonumber \\
	& \qquad + \dfrac{\alpha - 2}{\alpha - 1} \gamma^{2} k \sqrt{k} \left( 4 \gamma L \sqrt{k} + \alpha \right) \left\lVert v_{k} - v_{k-1} \right\rVert ^{2} .
	\end{align*}
	Now we use relation \eqref{split:Lip} and apply Lemma~\ref{lem:quad} with $\left( a, b, c \right) := \Bigl( \nicefrac{1}{2}, \minus 2 \gamma, \nicefrac{2 \gamma}{L} \Bigr)$ to verify that for every $k \geq 1$
	\begin{align*}
		\MoveEqLeft \dfrac{1}{2} \left\lVert z_{k} - z_{k-1} \right\rVert ^{2} - 4 \gamma \left\langle z_{k} - z_{k-1}, F \left( z_{k} \right) - F \left( w_{k-1} \right) \right\rangle + 8 \gamma^{3} L \left\lVert v_{k} - v_{k-1} \right\rVert ^{2} \\
		&\geq \dfrac{1}{2} \left\lVert z_{k} - z_{k-1} \right\rVert ^{2} - 4 \gamma \left\langle z_{k} - z_{k-1}, F \left( z_{k} \right) - F \left( w_{k-1} \right) \right\rangle + \dfrac{2 \gamma}{L} \left\lVert F \left( z_{k} \right) - F \left( w_{k-1} \right) \right\rVert ^{2}\\
		&\geq 0.
	\end{align*}
	Combining the last two estimates, for every $ k \geq 1 $ one can easily conclude that
	\begin{align*}
		\MoveEqLeft \cG_{\lambda,k}
		\geq \dfrac{\alpha - 2}{4 \left( 3 \alpha - 2 \right)} \left\lVert 4 \lambda \left( z_{k} - z^{\ast} \right) + 2k \left( z_{k} - z_{k-1} \right) + \dfrac{2 \left( 3 \alpha - 2 \right)}{\alpha - 1} \gamma k v_{k} \right\rVert ^{2} \nonumber \\
			&\quad + \left( \alpha - 2 \right) \left( \dfrac{1}{3 \alpha - 2} - \dfrac{1}{4 \left( \alpha - 1 \right)} \right) k^{2} \left\lVert z_{k} - z_{k-1} \right\rVert ^{2}\\
			&\quad + 2 \left( \alpha - 1 \right) \lambda \left( 1 - \dfrac{4 \lambda}{3 \alpha - 2} \right) \left\lVert z_{k} - z^{\ast} \right\rVert ^{2} \nonumber \\
		&= \dfrac{\alpha - 2}{4 \left( 3 \alpha - 2 \right)} \left\lVert 4 \lambda \left( z_{k} - z^{\ast} \right) + 2k \left( z_{k} - z_{k-1} \right) + \dfrac{2 \left( 3 \alpha - 2 \right)}{\alpha - 1} \gamma k v_{k} \right\rVert ^{2} \nonumber \\
			&\quad + \dfrac{\left( \alpha - 2 \right) ^{2}}{4 \left( 3 \alpha - 2 \right) \left( \alpha - 1 \right) } k^{2} \left\lVert z_{k} - z_{k-1} \right\rVert ^{2} + 2 \left( \alpha - 1 \right) \lambda \left( 1 - \dfrac{4 \lambda}{3 \alpha - 2} \right) \left\lVert z_{k} - z^{\ast} \right\rVert ^{2},
	\end{align*}
	which is the desired inequality.
	\qedhere
\end{proof}

The following lemma will be helpful in the main proof.
\begin{lemma}
	\label{lem:trunc}
	The following statements are true:
	\begin{enumerate}[(i)]
		\item there exist two parameters
		\begin{equation}
		\label{trunc:fea}
		0 \leq \underline{\lambda} \left( \alpha \right) < \overline{\lambda} \left( \alpha \right) \leq \dfrac{3 \alpha - 2}{4}
		\end{equation}
		such that for every $\lambda$ satisfying $\underline{\lambda} \left( \alpha \right) < \lambda < \overline{\lambda} \left( \alpha \right)$ one can find an integer $k_{\lambda} \geq 1$ with the property that the following inequality holds for every $k \geq k_{\lambda}$
		\begin{equation}\label{Rk}
		\begin{split}
			\MoveEqLeft R_{k} := \sqrt{\dfrac{5 \alpha - 2}{2 \left( 3 \alpha - 2 \right)}} \Bigl( \eta_{2} k + \kappa_{0} \sqrt{k} \Bigr) \left\lVert z_{k+1} - z_{k} \right\rVert ^{2}
			+ 4 \gamma \Bigl( \eta_{0} k + \eta_{1} \Bigr) \left\langle z_{k+1} - z_{k}, v_{k+1} \right\rangle \\
			&\quad + 4 \sqrt{\dfrac{5 \alpha - 2}{2 \left( 3 \alpha - 2 \right)}} \gamma^{2} \Bigl( \eta_{3} k + \kappa_{1} \sqrt{k} \Bigr) \left\lVert v_{k+1} \right\rVert ^{2}
			\leq 0;
		\end{split}
		\end{equation}

		\item
		there exists a positive integer $k_{\varepsilon}$ such that for every $k \geq k_{\varepsilon}$ we have
		\begin{equation}
		\label{muk}
		\mu_{k} \geq \dfrac{\varepsilon}{2} \left( k+1 \right)^{2} .
		\end{equation}
	\end{enumerate}
\end{lemma}
\begin{proof}
	(i) For the quadratic expression in $R_{k}$ we calculate
	\begin{align*}
		\dfrac{\Delta_{k}'}{4 \gamma^{2}}
		&:= \left( \eta_{0} k + \eta_{1} \right) ^{2} - \dfrac{5 \alpha - 2}{2 \left( 3 \alpha - 2 \right)} k \Bigl( \eta_{2} \sqrt{k} + \kappa_{0} \Bigr) \Bigl( \eta_{3} \sqrt{k} + \kappa_{1} \Bigr) \nonumber \\
		&= \left( \eta_{0}^{2} - \dfrac{5 \alpha - 2}{2 \left( 3 \alpha - 2 \right)} \eta_{2} \eta_{3} \right) k^{2} - \dfrac{5 \alpha - 2}{2 \left( 3 \alpha - 2 \right)} \left( \eta_{2} \kappa_{1} + \kappa_{0} \eta_{3} \right) k \sqrt{k} \\
			&\quad+ \left( 2 \eta_{0} \eta_{1} - \dfrac{5 \alpha - 2}{2 \left( 3 \alpha - 2 \right)} \kappa_{0} \kappa_{1} \right) k + \eta_{1}^{2} .
	\end{align*}
	Since $\left( \eta_{0}^{2} - \frac{5 \alpha - 2}{2 \left( 3 \alpha - 2 \right)} \eta_{2} \eta_{3} \right) k^{2}$ is the dominant term in the above polynomial, it suffices to guarantee that
	\begin{equation}
	\label{nonneg:gamma}
	\eta_{0}^{2} - \frac{5 \alpha - 2}{2 \left( 3 \alpha - 2 \right)} \eta_{2} \eta_{3} < 0
	\end{equation}
	holds in order to ensure the existence of some integer $k_{\lambda} \geq 1$ such that $\Delta_{k}' \leq 0$ for every $k \geq k_{\lambda}$ and to obtain from here, due to Lemma~\ref{lem:quad}~\eqref{quad:vec}, that $R_{k} \leq 0$ for every $k \geq k_{\lambda}$.

	It remains to show that there exists a choice of $\lambda$ for which \eqref{nonneg:gamma} is true. We set $\xi := \lambda + 1 - \alpha \leq 0$ and get
	\begin{align*}
		\eta_{0}
		&= \dfrac{1}{2 \left( \alpha - 1 \right)} \left( 4 \left( \alpha - 1 \right) \left( \lambda + 1 - \alpha \right) - \alpha \left( \alpha - 2 \right) \right) \\
		&= \dfrac{1}{2 \left( \alpha - 1 \right)} \left( 4 \left( \alpha - 1 \right) \xi - \alpha \left( \alpha - 2 \right) \right), \\
		\eta_{2} \eta_{3} 	& = - \dfrac{2}{\alpha - 1} \left( \alpha - 2 \right) \left( 3 \alpha - 2 \right) \left( \lambda + 1 - \alpha \right) = - \dfrac{2}{\alpha - 1} \left( \alpha - 2 \right) \left( 3 \alpha - 2 \right) \xi.
	\end{align*}
	This means that we have to guarantee that there exists a choice for $\xi$ satisfying
	\begin{align*}
		\MoveEqLeft \eta_{0}^{2} - \frac{5 \alpha - 2}{2 \left( 3 \alpha - 2 \right)} \eta_{2} \eta_{3}\\
		& = \dfrac{1}{4 \left( \alpha - 1 \right) ^{2}} \left( \left( 4 \left( \alpha - 1 \right) \xi - \alpha \left(\alpha -2 \right) \right) ^{2} + 4 \left( 5 \alpha - 2 \right) \left( \alpha - 1 \right) \left( \alpha - 2 \right) \xi \right) \nonumber \\
		& = \dfrac{1}{4 \left( \alpha - 1 \right) ^{2}} \left( 16 \left( \alpha - 1 \right) ^{2} \xi^{2} + 4 \left( \alpha - 1 \right) \left( \alpha - 2 \right) \left( 3 \alpha - 2 \right) \xi + \alpha ^{2} \left( \alpha - 2 \right) ^{2} \right) < 0,
	\end{align*}
	which is nothing else than
	\begin{equation}
	\label{trunc:omega-a}
	16 \left( \alpha - 1 \right) ^{2} \xi^{2} + 4 \left( \alpha - 1 \right) \left( \alpha - 2 \right) \left( 3 \alpha - 2 \right) \xi + \alpha ^{2} \left( \alpha - 2 \right) ^{2} < 0 .
	\end{equation}
	A direct computation shows that
	\begin{align*}
	\Delta_{\xi}
	:= 16 \left( \alpha - 1 \right) ^{2} \left( \alpha - 2 \right) ^{2} \left( \left( 3 \alpha - 2 \right) ^{2} - 4 \alpha ^{2} \right)
	= 16 \left( \alpha - 1 \right) ^{2} \left( \alpha - 2 \right) ^{3} \left( 5 \alpha - 2 \right) > 0.
	\end{align*}
	Hence, in order to get \eqref{trunc:omega-a}, we have to choose $\xi$ between the two roots of the quadratic function arising in this formula, in other words
	\begin{align*}
	\xi_{1} \left( \alpha \right) & := \dfrac{1}{32 \left( \alpha - 1 \right) ^{2}} \left( - 4 \left( \alpha - 1 \right) \left( \alpha - 2 \right) \left( 3 \alpha - 2 \right) - \sqrt{\Delta_{\xi}} \right) \nonumber \\
	& = - \dfrac{1}{8 \left( \alpha - 1 \right)} \left( \alpha - 2 \right) \left( 3 \alpha - 2  + \sqrt{\left( \alpha - 2 \right) \left( 5 \alpha - 2 \right)} \right) \nonumber \\
	& < \xi = \lambda + 1 - \alpha
	< \xi_{2} \left( \alpha \right) := \dfrac{1}{32 \left( \alpha - 1 \right) ^{2}} \left( - 4 \left( \alpha - 1 \right) \left( \alpha - 2 \right) \left( 3 \alpha - 2 \right) + \sqrt{\Delta_{\xi}} \right) \nonumber \\
	& = - \dfrac{1}{8 \left( \alpha - 1 \right)} \left( \alpha - 2 \right) \left(3 \alpha - 2 - \sqrt{\left( \alpha - 2 \right) \left( 5 \alpha - 2 \right)} \right) .
	\end{align*}
	Obviously $\xi_{1} \left( \alpha \right) < 0$ and from Vieta's formula $\xi_{1} \left( \alpha \right) \cdot \xi_{2} \left( \alpha \right) = \frac{\alpha^{2} \left(\alpha -2\right) ^{2}}{16 \left( \alpha - 1 \right) ^{2}}$, it follows that we must have $\xi_{2} \left( \alpha \right) < 0$ as well.

	Therefore, going back to $\lambda$, in order to be sure that $\eta_{0}^{2} - \frac{5\alpha-2}{2 \left( 3\alpha -2 \right)} \eta_{2} \eta_{3} < 0$ this must be chosen such that
	\begin{equation*}
	\alpha - 1 + \xi_{1} \left( \alpha \right) < \lambda < \alpha - 1 + \xi_{2} \left( \alpha \right) .
	\end{equation*}
	Next we will show that
	\begin{equation}
	\label{trunc:check}
	0 < \alpha - 1 - \dfrac{1}{8 \left( \alpha - 1 \right)} \left( \alpha - 2 \right) \left( 3 \alpha - 2 \right) < \dfrac{3 \alpha}{4} - \dfrac{1}{2}.
	\end{equation}
	Indeed, the inequality on the left-hand side follows immediately, since
	\begin{equation*}
	\begin{split}
		\MoveEqLeft \alpha - 1 - \dfrac{1}{8 \left( \alpha - 1 \right)} \left( \alpha - 2 \right) \left( 3 \alpha - 2 \right)
		= \dfrac{1}{8 \left( \alpha - 1 \right)} \left( 5 \alpha^{2} - 8 \alpha + 4 \right)\\
		&= \dfrac{1}{8 \left( \alpha - 1 \right)} \left( \alpha^{2} + 4 \left( \alpha - 1 \right) ^{2} \right)
		> 0 .
	\end{split}
	\end{equation*}
	Using this relation, one can notice that the inequality on the right hand side of \eqref{trunc:check} can be equivalently written as
	\begin{equation*}
	5 \alpha^{2} - 8 \alpha + 4 < 2 \left( \alpha - 1 \right) \left( 3 \alpha - 2 \right) \Leftrightarrow 0 < \alpha^{2} - 2 \alpha = \alpha \left( \alpha - 2 \right),
	\end{equation*}
	which is true as $\alpha >2$.

	From \eqref{trunc:check} we immediately deduce that
	\begin{equation*}
	0 < \alpha - 1 + \xi_{2} \left( \alpha \right)
	\quad \textrm{ and } \quad
	\alpha - 1 + \xi_{1} \left( \alpha \right) < \dfrac{3 \alpha}{4} - \dfrac{1}{2}.
	\end{equation*}
	This allows us to choose $ \underline{\lambda} < \overline{\lambda} $, where
	\begin{equation}
	\begin{split}
		\underline{\lambda} \left( \alpha \right)
			&:= \alpha - 1 + \xi_{1} \left( \alpha \right)\\
			&=  \dfrac{1}{8 \left( \alpha - 1 \right)} \alpha^{2} + \dfrac{1}{2} \left( \alpha - 1 \right) - \dfrac{1}{8 \left( \alpha - 1 \right)} \left( \alpha - 2 \right) \sqrt{\left( \alpha - 2 \right) \left( 5 \alpha - 2 \right)}  \nonumber \\
		\overline{\lambda} \left( \alpha \right)
			&:= \min \left\lbrace \dfrac{3 \alpha}{4} - \dfrac{1}{2},   \alpha - 1 + \xi_{2} \left( \alpha \right) \right\rbrace  \\
			&= \min \left\lbrace \dfrac{3 \alpha}{4} - \dfrac{1}{2}, \dfrac{1}{8 \left( \alpha - 1 \right)} \alpha^{2} + \dfrac{1}{2} \left( \alpha - 1 \right) + \dfrac{1}{8 \left( \alpha - 1 \right)} \left( \alpha - 2 \right) \sqrt{\left( \alpha - 2 \right) \left( 5 \alpha - 2 \right)} \right\rbrace,
	\end{split}
	\end{equation}
	since
	\begin{equation*}
	\dfrac{1}{8 \left( \alpha - 1 \right)} \alpha^{2} + \dfrac{1}{2} \left( \alpha - 1 \right) - \dfrac{1}{8 \left( \alpha - 1 \right)} \left( \alpha - 2 \right) \sqrt{\left( \alpha - 2 \right) \left( 5 \alpha - 2 \right)} > 0 .
	\end{equation*}
	Indeed, as $\left( \alpha - 1 \right) \sqrt{\alpha - 1} > \left( \alpha - 2 \right) \sqrt{\alpha - 2}$ and $4 \sqrt{\alpha - 1} > \sqrt{5 \alpha - 2}$ we can easily deduce that
	\begin{equation*}
	\alpha^{2} + 4 \left( \alpha - 1 \right) ^{2} > 4 \left( \alpha - 1 \right) ^{2} > \left( \alpha - 2 \right) \sqrt{\left( \alpha - 2 \right) \left( 5 \alpha - 2 \right)}
	\end{equation*}
	and the claim follows.

	In conclusion, choosing $\lambda$ to satisfy $\underline{\lambda} \left( \alpha \right) < \lambda < \overline{\lambda} \left( \alpha \right)$, we have $$\eta_{0}^{2} - \frac{5 \alpha - 2}{2 \left( 3 \alpha - 2 \right)} \eta_{2} \eta_{3} < 0$$ and therefore there exists some integer $k_{\lambda} \geq 1$ such that $R_k \leq 0$ for every $k_{\lambda}$.

	(ii) For every $k \geq 1$ we have
	\begin{align*}
	\mu_{k} - \dfrac{\varepsilon}{2} \left( k+1 \right) ^{2}
	= & \dfrac{\varepsilon}{2} \left( k+1 \right) ^{2} + \alpha^{2} \left( k+1 \right) \sqrt{k+1} + \left( \alpha - 4 \right) \left( k+1 \right) - \left( \alpha - 2 \right),
	\end{align*}
	and the conclusion is obvious.
	\qedhere
\end{proof}

The following proposition plays a key role in proving the convergence rates in Proposition~\ref{thm:o-rates} which will be used to prove Theorem~\ref{thm:conv}.
\begin{proposition}
	\label{prop:split:lim}
	Let $z^{\ast} \in \Omega$ and $ (z_{k})_{k \geq 0} $, $ (w_{k})_{k \geq 0} $, $ (\zeta_{k})_{k \geq 0} $ be the sequences generated by Algorithm~\ref{algo:split} and let $ (v_{k})_{k \geq 0} $ be the sequence defined by~\eqref{eq:vk}. Then the following statements are true:
	\begin{enumerate}[(i)]
		\item
		\label{split:lim:i}
		the following hold:
		\begin{subequations}
			\label{split:lim:sum}
			\begin{align}
			\sum_{k \geq 1} \left\langle z_{k} - z^{\ast}, F (z_{k}) + \zeta_{k} \right\rangle & < + \infty, \label{split:lim:vi} \\
			\sum_{k \geq 1} k^{2} \left\lVert v_{k+1} - v_{k} \right\rVert ^{2} & < + \infty, \label{split:lim:dV} \\
			\sum_{k \geq 1} k \left\lVert z_{k+1} - z_{k} \right\rVert ^{2} & < + \infty, \label{split:lim:dz} \\
			\sum_{k \geq 1} k \left\lVert F(w_{k}) + \zeta_{k+1} \right\rVert ^{2} & < + \infty ; \label{split:lim:V}
			\end{align}
		\end{subequations}

		\item
		\label{split:lim:ii}
		the sequence $\left(z_{k} \right) _{k \geq 0}$ is bounded and the following hold as $ k \to +\infty $:
		\begin{gather}
			\left\lVert z_{k} - z_{k-1} \right\rVert = \mathcal{O} \left( \dfrac{1}{k} \right),
			\quad
			\left\lVert \zeta_{k} +  F \left( w_{k-1} \right) \right\rVert = \mathcal{O} \left( \dfrac{1}{k} \right),
			\quad
			\left\lVert \zeta_{k} + F \left( z_{k} \right) \right\rVert = \mathcal{O} \left( \dfrac{1}{k} \right),
			\\
			\left\langle z_{k} - z^{\ast}, \zeta_{k} + F \left( z_{k} \right) \right\rangle = \mathcal{O} \left( \dfrac{1}{k} \right),
			\quad
			\left\langle z_{k} - z^{\ast}, F \left( z_{k} \right) \right\rangle = \mathcal{O} \left( \dfrac{1}{k} \right);
			\quad
		\end{gather}

		\item
		\label{split:lim:iii}
		there exist $0 \leq \underline{\lambda} \left( \alpha \right) < \overline{\lambda} \left( \alpha \right) \leq \nicefrac{(3 \alpha - 2)}{4}$ such that for every $\underline{\lambda} \left( \alpha \right) < \lambda < \overline{\lambda} \left( \alpha \right)$ the sequences $\left( \cE_{\lambda,k} \right)_{k \geq 1}$ and $\left( \cG_{\lambda,k} \right) _{k \geq 2}$ converge.
	\end{enumerate}
\end{proposition}

\begin{proof}
	According to Lemma~\ref{lem:trunc} there exist $\underline{\lambda} \left( \alpha \right) < \overline{\lambda} \left( \alpha \right)$ such that \eqref{trunc:fea} holds. We choose $\underline{\lambda} \left( \alpha \right) < \lambda < \overline{\lambda} \left( \alpha \right)$ and get, according to the same result, an integer $k_{\lambda} \geq 1$ such that for every $k \geq k_{\lambda}$ the inequality \eqref{Rk} holds. In addition, according to Lemma~\ref{lem:trunc}(ii), we get a positive integer $k_{\varepsilon}$ such that \eqref{muk} holds for every $k \geq k_{\varepsilon}$.

	This means that for every $k \geq k_{1} := \max \left\lbrace k_{0}, k_{\lambda}, k_{\varepsilon} \right\rbrace$, where $k_{0}$ is the positive integer provided by Lemma \ref{lem:reg}(i), we have
	\begin{equation}
	\begin{split}
		\MoveEqLeft \cG_{\lambda,k+1} - \cG_{\lambda,k}\\
			&\leq \dfrac{\left( \alpha - 1 \right) \left( \alpha - 2 \right) \lambda ^{2}}{\varepsilon \left( k+1 \right) ^{2}} \left\lVert z_{k+1} - z^{\ast} \right\rVert ^{2}
				- 4 \left( \alpha - 2 \right) \lambda \gamma \left\langle z_{k+1} - z^{\ast}, \zeta_{k+1} + F \left( z_{k+1} \right) \right\rangle \\
				&- \dfrac{\alpha - 2}{2 \left( \alpha - 1 \right)} \varepsilon \gamma^{2} \left( k+1 \right)^{2} \left\lVert v_{k+1} - v_{k} \right\rVert ^{2}
				+ \left[ \left( 1 - \sqrt{\dfrac{5 \alpha - 2}{2 \left( 3 \alpha - 2 \right)}} \right) \eta_{2} k + \kappa_{0} \sqrt{k} \right]  \left\lVert z_{k+1} - z_{k} \right\rVert ^{2} \\
				& + \left[\left( 1 - \sqrt{\dfrac{5 \alpha - 2}{2 \left( 3 \alpha - 2 \right)}} \right) \eta_{3} k + \kappa_{1} \sqrt{k} \right] 4 \gamma^{2}  \left\lVert v_{k+1} \right\rVert ^{2} .
	\end{split}
	\end{equation}
	Since $\eta_{2}, \eta_{3} < 0$ and $\kappa_{0}, \kappa_{1} \geq 0$, we can find some $ k_{2} \geq k_{1} $ large enough such that for every $ k \geq k_{2} $ we get
	\begin{equation}\label{ineqFk}
	\begin{split}
		\MoveEqLeft \cG_{\lambda,k+1}
			\leq \cG_{\lambda,k} + \dfrac{\left( \alpha - 1 \right) \left( \alpha - 2 \right) \lambda ^{2}}{\varepsilon \left( k+1 \right) ^{2}} \left\lVert z_{k+1} - z^{\ast} \right\rVert ^{2}
				- 4 \left( \alpha - 2 \right) \lambda \gamma \left\langle z_{k+1} - z^{\ast}, \zeta_{k+1} + F \left( z_{k+1} \right) \right\rangle \\
				&- \dfrac{\alpha - 2}{2 \left( \alpha - 1 \right)} \varepsilon \gamma^{2} \left( k+1 \right)^{2} \left\lVert v_{k+1} - v_{k} \right\rVert ^{2}
				+ \frac{1}{2} \left( 1 - \sqrt{\dfrac{5 \alpha - 2}{2 \left( 3 \alpha - 2 \right)}} \right) \eta_{2} k \left\lVert z_{k+1} - z_{k} \right\rVert ^{2}\\
				&+ \left( 1 - \sqrt{\dfrac{5 \alpha - 2}{2 \left( 3 \alpha - 2 \right)}} \right) 2 \gamma^{2} \eta_{3} k \left\lVert v_{k+1} \right\rVert ^{2} .
	\end{split}
	\end{equation}
	In view of \eqref{ene:low}, we get that $\cG_{\lambda,k} \geq 0$ for every $k \geq 1$ thus the sequence $\left( \cG_{\lambda,k}  \right) _{k \geq 2}$ is bounded from below.
	Moreover, by setting
	\begin{equation*}
	C_{0} := \dfrac{1}{2 \varepsilon} \left( \alpha - 2 \right) \lambda \left( 1 - \dfrac{4 \lambda}{3 \alpha - 2} \right) ^{-1} > 0,
	\end{equation*}
	we assert that
	\begin{align*}
	\MoveEqLeft \dfrac{\left( \alpha - 1 \right) \left( \alpha - 2 \right) \lambda ^{2}}{\varepsilon \left( k+1 \right) ^{2}} \left\lVert z_{k+1} - z^{\ast} \right\rVert ^{2}
	= \dfrac{C_{0}}{\left( k+1 \right) ^{2}} \cdot 2 \left( \alpha - 1 \right) \lambda \left( 1 - \dfrac{4 \lambda}{3 \alpha - 2} \right) \left\lVert z_{k} - z^{\ast} \right\rVert ^{2} \nonumber \\
	& \leq \dfrac{C_{0}}{\left( k+1 \right) ^{2}} \cG_{\lambda,k+1},
	\end{align*}
	Under these premises, we deduce from \eqref{ineqFk} that for every $k \geq k_{2} $
	\begin{equation}\label{inq:Fk-re}
	\begin{split}
		\MoveEqLeft \left( 1 - \dfrac{C_{0}}{\left( k+1 \right) ^{2}} \right) \cG_{\lambda,k+1}
		\leq \cG_{\lambda,k} - 4 \left( \alpha - 2 \right) \lambda \gamma \left\langle z_{k+1} - z^{\ast}, \zeta_{k+1} + F \left( z_{k+1} \right) \right\rangle \\
		&- \dfrac{\alpha - 2}{2 \left( \alpha - 1 \right)} \varepsilon \gamma^{2} \left( k+1 \right)^{2} \left\lVert v_{k+1} - v_{k} \right\rVert ^{2} \\
		&+ \frac{1}{2} \left( 1 - \sqrt{\dfrac{5 \alpha - 2}{2 \left( 3 \alpha - 2 \right)}} \right) \eta_{2} k \left\lVert z_{k+1} - z_{k} \right\rVert ^{2}
		+ \left( 1 - \sqrt{\dfrac{5 \alpha - 2}{2 \left( 3 \alpha - 2 \right)}} \right) 2 \gamma^{2} \eta_{3} k \left\lVert v_{k+1} \right\rVert ^{2} .
	\end{split}
	\end{equation}
	Taking $k_{3} := \max \left\lbrace k_{2}, \left\lceil \sqrt{C_{0}} - 1 \right\rceil \right\rbrace$ we conclude for every $ k \geq k_{3} $ that
	\begin{equation*}
	\left( 1 - \dfrac{C_{0}}{\left( k+1 \right) ^{2}} \right) ^{-1} = \dfrac{\left( k+1 \right) ^{2}}{\left( k+1 \right) ^{2} - C_{0}} = 1 + \dfrac{C_{0}}{\left( k+1 \right) ^{2} - C_{0}} > 1 .
	\end{equation*}
	Hence, for every $k \geq k_{3}$, the inequality \eqref{inq:Fk-re} leads to
	\begin{align*}
	\MoveEqLeft \cG_{\lambda,k+1}
	\leq \left( 1 + \dfrac{C_{0}}{\left( k+1 \right) ^{2} - C_{0}} \right) \cG_{\lambda,k} - 4 \left( \alpha - 2 \right) \lambda \gamma \left\langle z_{k+1} - z^{\ast}, \zeta_{k+1} + F \left( z_{k+1} \right) \right\rangle \\
	&- \dfrac{\alpha - 2}{2 \left( \alpha - 1 \right)} \varepsilon \gamma^{2} \left( k+1 \right)^{2} \left\lVert v_{k+1} - v_{k} \right\rVert ^{2} \\
	&+ \frac{1}{2} \left( 1 - \sqrt{\dfrac{5 \alpha - 2}{2 \left( 3 \alpha - 2 \right)}} \right) \eta_{2} k \left\lVert z_{k+1} - z_{k} \right\rVert ^{2}
	+ \left( 1 - \sqrt{\dfrac{5 \alpha - 2}{2 \left( 3 \alpha - 2 \right)}} \right)  2 \gamma^{2} \eta_{3} k \left\lVert v_{k+1} \right\rVert ^{2},
	\end{align*}
	which is nothing else than the inequality~\eqref{inq:G} with
	\begin{align*}
	b_{\lambda, k} 	& := 4 \left( \alpha - 2 \right) \lambda \gamma \left\langle z_{k+1} - z^{\ast}, \zeta_{k+1} + F \left( z_{k+1} \right) \right\rangle + \dfrac{\alpha - 2}{2 \left( \alpha - 1 \right)} \varepsilon \gamma^{2} \left( k+1 \right)^{2} \left\lVert v_{k+1} - v_{k} \right\rVert ^{2}
	\nonumber \\
	& \quad - \frac{1}{2} \left( 1 - \sqrt{\dfrac{5 \alpha - 2}{2 \left( 3 \alpha - 2 \right)}} \right) \eta_{2} k \left\lVert z_{k+1} - z_{k} \right\rVert ^{2}
	- \left( 1 - \sqrt{\dfrac{5 \alpha - 2}{2 \left( 3 \alpha - 2 \right)}} \right) 2 \gamma^{2} \eta_{3} k \left\lVert v_{k+1} \right\rVert ^{2}\\
	&\geq 0, \\
	d_{\lambda, k} 	& := \dfrac{C_{0}}{\left( k+1 \right) ^{2} - C_{0}} > 0 .
	\end{align*}
	Using Lemma~\ref{lem:quasi-Fej} we obtain \eqref{split:lim:sum} as well as convergence of the sequence $\left( \cG_{\lambda,k} \right)_{k \geq 1}$.

	Since $\left( \cG_{\lambda,k}  \right) _{k \geq 1}$ converges, it is also bounded from above, which, according to \eqref{ene:low}, implies that the following estimate holds for every $k \geq k_{3}$
	\begin{align*}
	\MoveEqLeft \dfrac{\alpha - 2}{3 \alpha - 2} \left\lVert 4 \lambda \left( z_{k} - z^{\ast} \right) + 2k \left( z_{k} - z_{k-1} \right) + \dfrac{2 \left( 3 \alpha - 2 \right)}{\alpha - 1} \gamma k v_{k} \right\rVert ^{2} \\
		&\quad + \dfrac{\left( \alpha - 2 \right) ^{2}}{4 \left( 3 \alpha - 2 \right) \left( \alpha - 1 \right)} k^{2} \left\lVert z_{k} - z_{k-1} \right\rVert ^{2}
		+ 2 \left( \alpha - 1 \right) \lambda \left( 1 - \dfrac{4 \lambda}{3 \alpha - 2} \right) \left\lVert z_{k} - z^{\ast} \right\rVert ^{2}\\
		&\leq \cG_{\lambda,k}
		\leq \sup_{k \geq 1} \cG_{\lambda,k}
		< +\infty .
	\end{align*}
	From here	we obtain the boundedness of the sequences
	\begin{equation}
	\begin{gathered}
	\left( 4 \lambda \left( z_{k} - z^{\ast} \right) + 2k \left( z_{k} - z_{k-1} \right) + \dfrac{2 \left( 3 \alpha - 2 \right)}{\alpha - 1} \gamma k v_{k} \right) _{k \geq 1},\\
	\left( k \left( z_{k} - z_{k-1} \right) \right) _{k \geq 1}
	\quad \textrm{and} \quad
	\left(z_{k} \right) _{k \geq 0} .
	\end{gathered}
	\end{equation}
	In particular, for every $k \geq k_{3}$ we have
	\begin{align}
		\left\lVert 4 \lambda \left( z_{k} - z^{\ast} \right) + 2k \left( z_{k} - z_{k-1} \right) + \dfrac{2 (3 \alpha - 2)}{\alpha - 1} \gamma k v_{k} \right\rVert
		\leq C_{1} := \sqrt{\dfrac{3 \alpha - 2}{\alpha - 2} \sup_{k \geq 1} \cG_{\lambda,k}} < + \infty, \\
		k \left\lVert z_{k} - z_{k-1} \right\rVert
		\leq C_{2} := \dfrac{2}{\alpha - 2} \sqrt{\left( 3 \alpha - 2 \right) \left( \alpha - 1 \right) \sup_{k \geq 1} \cG_{\lambda,k}} < + \infty, \\
		\left\lVert z_{k} - z^{\ast} \right\rVert
		\leq C_{3} := \sqrt{\dfrac{1}{2 \left( \alpha - 1 \right) \lambda} \left( 1 - \dfrac{4 \lambda}{3 \alpha - 2} \right) ^{-1} \sup_{k \geq 1} \cG_{\lambda,k}} < + \infty .\label{ineq-C3}
	\end{align}
	Using the triangle inequality, we deduce from here that  for every $k \geq k_{3}$
	\begin{align}
	\left\lVert v_{k} \right\rVert
	& \leq \dfrac{\alpha - 1}{2 \left( 3 \alpha - 2 \right) \gamma k} \left\lVert 4 \lambda \left( z_{k} - z^{\ast} \right) + 2k \left( z_{k} - z_{k-1} \right) + \dfrac{2 \left( 3 \alpha - 2 \right)}{\alpha - 1} \gamma k v_{k} \right\rVert \nonumber \\
	& \quad + \dfrac{\alpha - 1}{\left( 3 \alpha - 2 \right) \gamma} \left\lVert z_{k} - z_{k-1} \right\rVert + \dfrac{2 \left( \alpha - 1 \right) \lambda}{\left( 3 \alpha - 2 \right) \gamma k} \left\lVert z_{k} - z^{\ast} \right\rVert \leq \dfrac{C_{4}}{k},
	\end{align}
	where
	\begin{equation*}
	C_{4} := \dfrac{\alpha - 1}{2 \left( 3 \alpha - 2 \right) \gamma} \left( C_{1} + 2C_{2} + 4 \overline{\lambda} \left( \alpha \right) C_{3} \right) > 0 .
	\end{equation*}
	The statement \eqref{split:lim:dV} yields
	\begin{equation}
	\lim\limits_{k \to + \infty} k \left\lVert v_{k+1} - v_{k} \right\rVert = 0
	\quad \Rightarrow \quad
	C_{5} := \sup_{k \geq 1} \left\lbrace k \left\lVert v_{k+1} - v_{k} \right\rVert \right\rbrace < + \infty, \label{split:lim:dV-0}
	\end{equation}
	which, together with \eqref{split:Lip} implies that  for every $k \geq k_{3}$
	\begin{equation}
	\begin{split}
	\left\lVert \zeta_{k+1} + F \left( z_{k+1} \right) \right\rVert
	&\leq \left\lVert \zeta_{k+1} + F \left( z_{k+1} \right) - v_{k+1} \right\rVert + \left\lVert v_{k+1} \right\rVert\\
	&\leq \left\lVert v_{k+1} - v_{k} \right\rVert + \left\lVert v_{k+1} \right\rVert \label{split:lim:V-dV}
	\leq \dfrac{C_{6}}{k},
	\end{split}
	\end{equation}
	where
	\begin{equation*}
	C_{6} := C_{4} + C_{5} > 0 .
	\end{equation*}
	The remaining assertion follows from the fact that $\zeta_{k} \in N_{C} \left( z_{k} \right)$ where $z_{k} \in C$ by definition, the Cauchy-Schwarz inequality and the boundedness of $\left( z_{k} \right) _{k \geq 0}$, namely, for every $k \geq k_{3}$ we deduce
	\begin{equation*}
	0 \leq \left\langle z_{k} - z^{\ast}, F \left( z_{k} \right) \right\rangle \leq \left\langle z_{k} - z^{\ast}, \zeta_{k} + F \left( z_{k} \right) \right\rangle \leq \left\lVert z_{k} - z^{\ast} \right\rVert \left\lVert \zeta_{k} + F \left( z_{k} \right) \right\rVert \leq \dfrac{C_{3} C_{6}}{k-1} .
	\end{equation*}
	To complete the proof, we are going to show that in fact
	\begin{equation*}
	\lim\limits_{k \to + \infty} \cE_{\lambda,k} = \lim\limits_{k \to + \infty} \cG_{\lambda,k}  \in \R .
	\end{equation*}
	Indeed, we already have seen that
	\begin{equation*}
	\lim\limits_{k \to + \infty} \left( k + 1 \right) \left\lVert v_{k+1} - v_{k} \right\rVert = \lim\limits_{k \to + \infty} \left\lVert v_{k+1} \right\rVert = 0,
	\end{equation*}
	which, by the Cauchy-Schwarz inequality and \eqref{split:Lip} yields
	\begin{align*}
	0 \leq \lim\limits_{k \to + \infty} k^{2} \left\lvert \left\langle z_{k} - z_{k-1}, F \left( z_{k} \right) - F \left( w_{k-1} \right) \right\rangle \right\rvert
	& \leq C_{2} \lim\limits_{k \to + \infty} k \left\lVert F \left( z_{k} \right) - F \left( w_{k-1} \right) \right\rVert \nonumber \\
	& \leq C_{2} \lim\limits_{k \to + \infty} k \left\lVert v_{k} - v_{k-1} \right\rVert = 0 .
	\end{align*}
	From here we obtain the desired statement.
\end{proof}

\subsection{Proofs of the Main Results}

\begin{proof}[Proof of Theorem~\ref{thm:conv}]
	Let $\underline{\lambda} \left( \alpha \right) < \overline{\lambda} \left( \alpha \right)$ be the parameters provided by Lemma~\ref{lem:trunc} such that \eqref{trunc:fea} holds and with the property that for every $\underline{\lambda} \left( \alpha \right) < \lambda < \overline{\lambda} \left( \alpha \right)$ there exists an integer $k_{\lambda} \geq 1$ such that for every $k \geq k_{\lambda}$ the inequality \eqref{Rk} holds.

	For every $k \geq 1$ we set
	\begin{align}
	p_{k} &:= \dfrac{1}{2} \left( \alpha - 1 \right) \left\lVert z_{k} - z^{\ast} \right\rVert ^{2} + k \left\langle z_{k} - z^{\ast}, z_{k} - z_{k-1} + 2 \gamma v_{k} \right\rangle, \label{split:defi:p-k} \\
	q_{k}    & := \dfrac{1}{2} \left\lVert z_{k} - z^{\ast} \right\rVert ^{2} + 2 \gamma \sum_{i = 1}^{k} \left\langle z_{i} - z^{\ast}, v_{i} \right\rangle . \label{split:defi:q-k}
	\end{align}
	Then one can see that for every $k \geq 2$ we have
	\begin{align*}
	q_{k} - q_{k-1}
	& = \left\langle z_{k} - z^{\ast}, z_{k} - z_{k-1} \right\rangle - \dfrac{1}{2} \left\lVert z_{k} - z_{k-1} \right\rVert ^{2} + 2 \gamma \left\langle z_{k} - z^{\ast}, v_{k} \right\rangle,
	\end{align*}
	and thus
	\begin{equation*}
	\left( \alpha - 1 \right) q_{k} + k \left( q_{k} - q_{k-1} \right) = p_{k} + 2 \left( \alpha - 1 \right) \gamma \sum_{i = 1}^{k} \left\langle z_{i} - z^{\ast}, v_{i} \right\rangle - \dfrac{k}{2} \left\lVert z_{k} - z_{k-1} \right\rVert ^{2} .
	\end{equation*}

	From \eqref{split:defi:E-k} and \eqref{split:defi:u-k-1-lambda}, direct computation shows that for every $k \geq 1$
	\begin{equation}\label{defi:E-lambda}
	\begin{split}
		\cE_{\lambda,k}
		&= \dfrac{1}{2} \left\lVert 2 \lambda \left( z_{k} - z^{\ast} \right) + 2k \left( z_{k} - z_{k-1} \right) + \dfrac{3 \alpha - 2}{\alpha - 1} \gamma k v_{k} \right\rVert ^{2}
		+ 2 \lambda \left( \alpha - 1 - \lambda \right) \left\lVert z_{k} - z^{\ast} \right\rVert ^{2} \\
			&\quad + \dfrac{2 \left( \alpha - 2 \right)}{\alpha - 1} \lambda \gamma k \left\langle z_{k} - z^{\ast}, v_{k} \right\rangle + \dfrac{\alpha - 2}{\alpha - 1} \gamma^{2} k \left( \dfrac{1}{2 \left( \alpha - 1 \right)} \left( 3 \alpha - 2 \right) k + \alpha \right) \left\lVert v_{k} \right\rVert ^{2} \\
		&= 2 \lambda \left( \alpha - 1 \right) \left\lVert z_{k} - z^{\ast} \right\rVert ^{2} + 4 \lambda k \left\langle z_{k} - z^{\ast}, z_{k} - z_{k-1} + 2 \gamma v_{k} \right\rangle + \dfrac{\alpha - 2}{\alpha - 1} \alpha \gamma^{2} k \left\lVert v_{k} \right\rVert ^{2} \\
			&\quad + \dfrac{k^{2}}{2} \left( \left\lVert 2 \left( z_{k} - z_{k-1} \right) + \dfrac{3 \alpha - 2}{\alpha - 1} \gamma v_{k} \right\rVert ^{2} + \dfrac{\left( \alpha - 2 \right) \left( 3 \alpha - 2 \right)}{\left( \alpha - 1 \right)^{2}} \gamma^{2} \left\lVert v_{k} \right\rVert ^{2} \right) .
	\end{split}
	\end{equation}
	Therefore, for every $\underline{\lambda} \left( \alpha \right) < \lambda_{1} < \lambda_{2} < \overline{\lambda} \left( \alpha \right)$ we can conclude
	\begin{align*}
	\cE_{\lambda_{2},k} - \cE_{\lambda_{1},k}
	= \ & 4 \left( \lambda_{2} - \lambda_{1} \right) \left( \dfrac{1}{2} \left( \alpha - 1 \right) \left\lVert z_{k} - z^{\ast} \right\rVert ^{2} + 2  k \left\langle z_{k} - z^{\ast}, \left( z_{k} - z_{k-1} \right) +
	s v_{k} \right\rangle \right) \nonumber \\
	= \ & 4 \left( \lambda_{2} - \lambda_{1} \right) p_{k} .
	\end{align*}
	Hence, according to the previous theorem, the limit $\lim_{k \to + \infty} \left(\cE_{\lambda_{2},k} - \cE_{\lambda_{1},k} \right)\in \R$ exists, which implies further that the limit
	\begin{equation}
	\label{split:lim-p-k}
	\lim\limits_{k \to + \infty} p_{k} \in \R \textrm{ exists}.
	\end{equation}
	Further, we observe that for every $k \geq 2$
	\begin{subequations}
		\begin{align}
		\sum_{i = 2}^{k} \left\lvert \left\langle z_{i} - z^{\ast}, F \left( w_{i-1} \right) - F \left( z_{i} \right) \right\rangle \right\rvert
		& \leq \sum_{i = 2}^{k} \left\lVert z_{i} - z^{\ast} \right\rVert \left\lVert F \left( w_{i-1} \right) - F \left( z_{i} \right) \right\rVert \label{conv:inn:C-S} \\
		& \leq \frac{1}{2} \sum_{i=2}^{k} \frac{1}{i^{2}} \normsq{z_{i} - z^{\ast}} + \frac{1}{2} \sum_{i=2}^{k} i^{2} \normsq{F(w_{i-1}) - F(z_{i})} \\
		& \leq \frac{1}{2} \sum_{i=2}^{+ \infty} \frac{1}{i^{2}} \normsq{z_{i} - z^{\ast}} + \frac{1}{2} \sum_{i=2}^{+ \infty} i^{2} \normsq{F(w_{i-1}) - F(z_{i})} < + \infty, \label{conv:inn:sup}
		\end{align}
	\end{subequations}
	where \eqref{conv:inn:C-S} comes from the Cauchy-Schwarz inequality, the first sum in~\eqref{conv:inn:sup} is finite due to~\eqref{ineq-C3}, while the second series is convergent because of~\eqref{split:Lip} and~\eqref{split:lim:dV}. This means the series $\sum_{k \geq 2} \left \langle z_{k} - z^{\ast}, F \left( w_{k-1} \right) - F \left( z_{k} \right) \right\rangle$ is absolutely convergent, thus convergent.

	By taking into consideration \eqref{split:lim:vi}, it follows from here that the limit
	\begin{equation*}
	\begin{split}
		\MoveEqLeft \lim_{k \to + \infty} \sum_{i = 1}^{k} \left\langle z_{i} - z^{\ast}, v_{i} \right\rangle\\
		&= \lim_{k \to + \infty} \sum_{i = 1}^{k} \left\langle z_{i} - z^{\ast}, \zeta_{i} + F \left( z_{i} \right) \right\rangle + \lim_{k \to + \infty} \sum_{i = 1}^{k} \left\langle z_{i} - z^{\ast}, F \left( w_{i-1} \right) - F \left( z_{i} \right) \right\rangle \in \R
	\end{split}
	\end{equation*}
	exists. In addition, thanks to \eqref{split:lim:dz}, we have $\lim_{k \to + \infty} k \left\lVert z_{k+1} - z_{k} \right\rVert ^{2} = 0$, consequently,
	\begin{equation*}
	\lim\limits_{k \to + \infty} \left( \alpha - 1 \right) q_{k} + k \left( q_{k} - q_{k-1} \right) \in \R \textrm{ exists} .
	\end{equation*}
	According to Proposition~\ref{prop:split:lim}, we have that $\left(q_{k} \right)_{k \geq 1}$ is bounded due to the boundedness of $\left(z_{k} \right)_{k \geq 0}$ and the fact that $\lim_{k \to + \infty} \sum_{i = 1}^{k} \left\langle z_{i} - z^{\ast}, v_{i} \right\rangle \in \R$ exists.
	Therefore, we can apply Lemma~\ref{lem:lim-u-k} to guarantee the existence of the limit $\lim_{k \to + \infty} q_{k} \in \R$.
	By the definition of $q_{k}$ in \eqref{split:defi:q-k} and the fact that the sequence $\left( \sum_{i = 1}^{k-1} \left\langle z_{i} - z^{\ast}, v_{i} \right\rangle \right)_{k \geq 1}$ converges, we conclude that $\lim_{k \to + \infty} \left\lVert z_{k} - z^{\ast} \right\rVert \in \R$ exists. The hypothesis (i) in the Opial Lemma (see Lemma~\ref{lem:opial}) is fulfilled.

	Let $w$ be a cluster point of $\left(z_{k} \right)_{k \geq 0}$, which means that there exists a subsequence $\left\lbrace z_{k_{n}} \right\rbrace _{n \geq 0}$ such that
	\begin{equation*}
	z_{k_{n}} \to w \textrm{ as } n \to + \infty .
	\end{equation*}
	It follows from Proposition~\ref{prop:split:lim} that
	\begin{equation*}
	 F \left( z_{k_{n}} \right) + \zeta_{k_{n}} \to 0
	 \quad \textrm{as } n \to + \infty .
	\end{equation*}
	The maximal monotonicity of $F + N_{C}$ implies that $0 \in \left( N_{C} + F \right) \left( w \right)$, meaning that hypothesis (ii) of Lemma~\ref{lem:opial} is also verified. The proof of the convergence of the iterates is therefore completed.
\end{proof}

Before finally proving Theorem~\ref{thm:gap-res} we show convergence rates of various helpful quantities.
\begin{proposition}\label{thm:o-rates}
	Let $ z^{\ast} \in \Omega $ and $(z_{k})_{k \geq 0}$ be the sequence generated by Algorithm~\ref{algo:split}. Then, as $ k \to +\infty $, the following hold:
	\begin{equation}
	\begin{gathered}
		\norm{z_{k} - z_{k-1}} = o \left( \dfrac{1}{k} \right), \quad
		\sprod{F(z_{k})}{z_{k} - z^{\ast}} = o \left( \frac{1}{k} \right), \quad
		\sprod{\zeta_{k} + F(z_{k})}{z_{k} - z^{\ast}} = o \left( \frac{1}{k} \right)\\
		\norm{\zeta_{k} + F(z_{k})} = o \left( \frac{1}{k} \right), \quad
		\norm{\zeta_{k} +  F(w_{k-1})} = o \left( \frac{1}{k} \right).
	\end{gathered}
	\end{equation}
\end{proposition}

\begin{proof}
	Let $\underline{\lambda} \left( \alpha \right) < \overline{\lambda} \left( \alpha \right)$ be the parameters provided by Lemma~\ref{lem:trunc} such that \eqref{trunc:fea} holds and with the property that for every $\underline{\lambda} \left( \alpha \right) < \lambda < \overline{\lambda} \left( \alpha \right)$ there exists an integer $k_{\lambda} \geq 1$ such that for every $k \geq k_{\lambda}$ the inequality \eqref{Rk} holds.
	We fix $\underline{\lambda} \left( \alpha \right) < \lambda < \overline{\lambda} \left( \alpha \right)$ and recall that according to Proposition~\ref{prop:split:lim}(iii) the sequence $(\cE_{\lambda,k} )_{k \geq 1}$ converges.

	We set for every $k \geq 1$
	\begin{equation*}
		h_{k} := \dfrac{k^{2}}{2} \left( \left\lVert 2 \left( z_{k} - z_{k-1} \right) + \dfrac{3 \alpha - 2}{\alpha - 1} \gamma v_{k} \right\rVert ^{2} + \dfrac{\left( \alpha - 2 \right) \left( 3 \alpha - 2 \right)}{\left( \alpha - 1 \right) ^{2}} \gamma^{2} \left\lVert v_{k} \right\rVert ^{2} \right),
	\end{equation*}
	and notice that, in view of \eqref{defi:E-lambda} and \eqref{split:defi:p-k}, we have
	\begin{equation*}
	\cE_{\lambda,k} = 4 \lambda p_{k} +  \dfrac{4 \left( \alpha - 2 \right)}{\alpha - 1} \alpha \gamma^{2} \gamma^{2} k \left\lVert v_{k} \right\rVert ^{2} + h_{k} .
	\end{equation*}
	Proposition~\ref{prop:split:lim} asserts that
	\begin{equation*}
	\lim\limits_{k \to \infty} k \left\lVert v_{k} \right\rVert ^{2} = 0,
	\end{equation*}
	which, together with $\lim_{k \to + \infty} \cE_{\lambda,k} \in \R$ and $\lim_{k \to + \infty} p_{k} \in \R$ (see also \eqref{split:lim-p-k}), yields the existence of
	\begin{equation*}
	\lim\limits_{k \to + \infty} h_{k} \in \R.
	\end{equation*}
	In addition, \eqref{split:lim:dz} and \eqref{split:lim:V} in Proposition~\ref{prop:split:lim} guarantee that
	\begin{align*}
	\sum_{k \geq 1} \dfrac{1}{k} h_{k} \leq 4 \sum_{k \geq 1} k \left\lVert z_{k} - z_{k-1} \right\rVert ^{2} + \dfrac{\left( 3 \alpha - 2 \right) \left( 7 \alpha - 6 \right)}{2 \left( \alpha - 1 \right) ^{2}} \gamma^{2} \sum_{k \geq 1} k \left\lVert v_{k} \right\rVert ^{2} < + \infty .
	\end{align*}
	Consequently, $\lim_{k \to + \infty} h_{k} = 0$, which yields
	\begin{equation*}
	\lim\limits_{k \to \infty} k \left\lVert 2 \left( z_{k} - z_{k-1} \right) + \dfrac{3 \alpha - 2}{\alpha - 1} \gamma v_{k} \right\rVert = \lim\limits_{k \to \infty} k \left\lVert v_{k} \right\rVert = 0 .
	\end{equation*}
	This immediately implies $\lim_{k \to + \infty} k\left\lVert z_{k} - z_{k-1} \right\rVert = 0$. The fact that
	\begin{equation}
		\lim_{k \to + \infty} k \left\lVert \zeta_{k} + F \left( z_{k} \right) \right\rVert = 0
	\end{equation}
	follows from \eqref{split:Lip}, \eqref{split:lim:dV-0} and \eqref{split:lim:V-dV}, since
	\begin{equation*}
	0 \leq \lim\limits_{k \to + \infty} k \left\lVert \zeta_{k} + F \left( z_{k} \right) \right\rVert
	\leq \lim\limits_{k \to + \infty} k \left\lVert v_{k} - v_{k-1} \right\rVert + \lim\limits_{k \to + \infty} k \left\lVert v_{k} \right\rVert = 0 .
	\end{equation*}
	Finally, using the Cauchy-Schwarz inequality and the fact that $\left(z_{k} \right) _{k \geq 0}$ is bounded, we obtain that $\lim_{k \to + \infty} k \left\langle z_{k} - z^{\ast}, F \left( z_{k} \right) \right\rangle = \lim_{k \to + \infty} k \left\langle z_{k} - z^{\ast}, \zeta_{k} + F \left( z_{k} \right) \right\rangle = 0$.
\end{proof}

Now we are able to prove the convergence rates in terms of the restricted gap and the natural gap.
\begin{proof}[Proof of Theorem~\ref{thm:gap-res}]
	For every $k \geq 1$, using successively the monotonicity of $F$, the fact that $\zeta_{k} \in N_{C} (z_{k})$, where $z_{k} \in C$ by its definition, and the Cauchy-Schwarz inequality, we deduce that for every $u \in C \cap \mathbb{B} (z^{\ast} ; \delta(z_{0}))$
	\begin{equation}
	\begin{split}
		\MoveEqLeft \sprod{F(u)}{z_{k} - u}
		\leq \sprod{F(z_{k})}{z_{k} - u}
		\leq \sprod{\zeta_{k} + F(z_{k})}{z_{k} - u} \\
		&= \sprod{\zeta_{k} + F(z_{k})}{z_{k} - z^{\ast}} + \sprod{\zeta_{k} + F(z_{k})}{z^{\ast} - u} \\
		&\leq \sprod{\zeta_{k} + F(z_{k})}{z_{k} - z^{\ast}} + \norm{\zeta_{k} + F(z_{k})} \norm{z^{\ast} - u} \\
		&\leq \sprod{\zeta_{k} + F(z_{k})}{z_{k} - z^{\ast}} + \delta(z_{0}) \norm{\zeta_{k} + F(z_{k})}.
	\end{split}
	\end{equation}
	Therefore, it follows from Proposition~\ref{thm:o-rates} that
	\begin{equation}
	\begin{split}
		\Gap(z_{k})
		&= \max_{u \in C \cap \mathbb{B}(z^{\ast}; \delta(z_{0}))} \sprod{F(u)}{z_{k} - u}
		\leq \sprod{\zeta_{k} + F(z_{k})}{z_{k} - z^{\ast}} + \delta(z_{0}) \norm{\zeta_{k} + F(z_{k})} \\
		&= o \left( \frac{1}{k} \right)
		\quad \text{as } k \to +\infty .
	\end{split}
	\end{equation}

	Concluding, by~\eqref{eq:gaps} we obtain
	\begin{equation}
		\Res(z_{k})
		\leq \left\lVert \zeta_{k} + F(z_{k}) \right\rVert
		= o \left( \frac{1}{k} \right),
		\quad \textrm{as } k \to +\infty,
	\end{equation}
	and the proof is complete.
\end{proof}

\section{Implementation Details}
In this section we report the details on the implementations for our GAN experiments.

\subsection{Architecture}\label{app:resnet}
In Table~\ref{tab:resnet_architecture} we describe the architectures that were used in the experiments on CIFAR-10. The models were selected replicating the set-up of~\cite{SN-GAN,lookahead-minmax}.

\newpage
\begin{table}[ht]
	\caption{ResNet architecture used for the CIFAR-10 experiments.}%
	\label{tab:resnet_architecture}
	\vskip 0.15in
	\begin{center}
		\begin{small}
			\begin{tabular}{c}
				\toprule
				\textbf{Generator (G)}\\
				\toprule
				\textit{Input:} $z \in \R^{128} \sim \mathcal{N}(0, I) $ \\
				Linear $128 \to 4,096 $ \\
				G-ResBlock\\
				G-ResBlock\\
				G-ResBlock\\
				Batch Normalisation \\
				ReLU  \\
				conv. (kernel: $3{\times}3$, $256 \to 3$, stride: $1$, pad: 1) \\
				$\operatorname{tanh}(\ph)$\\

				\toprule
				\textbf{Discriminator (D)}\\
				\toprule
				\textit{Input:} $x \in \R^{3{\times}32{\times}32} $ \\
				D-ResBlock\\
				D-ResBlock\\
				D-ResBlock\\
				D-ResBlock\\
				ReLU\\
				Avg. Pool (kernel: $ 8 \times 8 $)
				Linear $128 \to 1$\\
				Spectral Normalisation\\
				\bottomrule
			\end{tabular}%
		\end{small}
	\end{center}
	\vskip -0.1in
\end{table}

\subsection{Hyperparameters}
In Table~\ref{tab:fogda-params} we list the hyperparameters that were used for fOGDA-VI to obtain the results on CIFAR-10. The hyperparameters for LA-GDA were the same as in~\cite{lookahead-minmax}.

\begin{table}[ht]
	\caption{Hyperparameters used for the GAN experiments on CIFAR-10.}%
	\label{tab:fogda-params}
	\vskip 0.15in
	\begin{center}
		\begin{small}
			\begin{tabular}{ll}
				\toprule
				\multicolumn{2}{l}{\textbf{fOGDA-VI}} \\
				\midrule
				Batch size &= $128$ \\
				Iterations &= $500,000$ \\
				Adam $\beta_1$ &= $0.0$ \\
				Adam $\beta_2$ &= $0.9$ \\
				Update ratio D/G &= $5$ \\
				Learning rate for discriminator &= $1\times10^{-4}$\\
				Learning rate for generator &= $1\times10^{-4}$\\
				fOGDA $ \alpha $ &= $ 100 $\\
				fOGDA $ n $ &= $ 1000 $\\
				\bottomrule
			\end{tabular}%
		\end{small}
	\end{center}
	\vskip -0.1in
\end{table}

\subsection{PyTorch Code}
In the following we report the code of the wrapper for the fOGDA-VI optimiser written using the PyTorch\\\cite{pytorch} framework.
\inputpython{fogda.py}{1}{150}


\end{document}